\documentclass{amsart}
\usepackage[margin=1.45in]{geometry}

\usepackage{amsmath}
\usepackage{amssymb}
\usepackage{amsthm}
\usepackage{amscd}
\usepackage{enumerate}
\usepackage{pinlabel}

\usepackage{tikz,tikz-cd}
\usetikzlibrary{matrix,arrows,decorations.pathmorphing}

\usepackage{hyperref}
\usepackage[margin=1.5cm,small]{caption} 

\theoremstyle{plain}

\newtheorem{thmspecial}{Theorem}
\newtheorem{corspecial}[thmspecial]{Corollary}

\newtheorem{thm}{Theorem}[section]
\newtheorem{prop}[thm]{Proposition}
\newtheorem{lem}[thm]{Lemma}

\theoremstyle{definition}
\newtheorem{defn}{Definition}
\newtheorem{notation}{Notation}
\theoremstyle{remark}
\newtheorem{remark}[thm]{Remark}

\newcommand{\FF}{\mathbb{F}}
\newcommand{\KK}{\mathbb{K}}
\DeclareMathOperator{\Hamm}{Hamm}
\DeclareMathOperator{\Sym}{Sym}
\DeclareMathOperator{\cyc}{cyc}
\DeclareMathOperator{\id}{Id}
\DeclareMathOperator{\Aut}{Aut}
\DeclareMathOperator{\rank}{Rank}
\DeclareMathOperator{\tr}{tr}
\DeclareMathOperator{\supp}{Supp}\newcommand{\Supp}{\supp}
\DeclareMathOperator{\proj}{proj}
\DeclareMathOperator{\rea}{re}

\newcommand{\modulus}[1]{\left|#1\right|}
\newcommand{\calC}{\mathcal{C}}
\newcommand{\dhs}{d_\mathrm{HS}}
\newcommand{\lhs}{\ell_\mathrm{HS}}
\newcommand{\calU}{\mathcal{U}}
\newcommand{\dbarhs}{\overline{d}_{\mathrm{HS}}}
\newcommand{\lbarhs}{\overline{\ell}_{\mathrm{HS}}}
\newcommand{\norm}[1]{\left\|#1\right\|_2}
\newcommand{\calH}{\mathcal{H}}
\newcommand{\drk}{d_{\mathrm{rk}}}
\newcommand{\lrk}{\ell_{\mathrm{rk}}}
\newcommand{\GL}{\operatorname{GL}}
\newcommand{\PGL}{\operatorname{PGL}}
\newcommand{\drkbar}{\overline{d}_{\mathrm{rk}}}
\newcommand{\lrkbar}{\overline{\ell}_{\mathrm{rk}}}
\newcommand{\m}{^{-1}}
\newcommand{\Z}{\mathbb{Z}}
\newcommand{\N}{\mathbb{N}}
\newcommand{\C}{\mathbb{C}}

\begin{document}
\title{Metric Approximations of Wreath Products}
\author{Ben Hayes and Andrew W. Sale}
\thanks{The first-named author gratefully acknowledges support by NSF Grant DMS-1600802.\\}

\begin{abstract} 
	 Given the large class of groups already known to be sofic, there is seemingly a shortfall in results concerning their permanence properties.
	We address this problem for wreath products, and in particular investigate the behaviour of more general metric approximations of groups under wreath products. 
	
	Our main result is the following. Suppose that $H$ is a sofic group and $G$ is a countable, discrete group. If $G$ is sofic, hyperlinear, weakly sofic, or linear sofic, then $G\wr H$ is also sofic, hyperlinear, weakly sofic, or linear sofic respectively. 
	In each case we construct relevant metric approximations, extending a general construction of metric approximations for $G\wr H$ that uses soficity of $H$. \end{abstract}

\maketitle

\section{Introduction}

Sofic groups, introduced by Gromov \cite{GromovSofic} and developed by Weiss \cite{WeissSofic}, are a large class of groups that can be approximated, in some sense, by finite groups. 
We consider sofic groups, as well as several other classes of groups which can be similarly defined by metric approximations, namely weakly sofic groups (introduced by Glebsky and Rivera \cite{Glebsky-Rivera}), linear sofic groups (introduced by Arzhantseva and Paunescu \cite{LinearSofic}) and hyperlinear groups (implicitly defined by Connes and explicitly by R\u{a}dulescu \cite{Connes,Radul}).

Via their approximations, sofic, hyperlinear, linear sofic, and weakly sofic groups have applications to a wide area of fields. 
For example, sofic groups are relevant to ergodic theory \cite{Bow,KLi}, 
topological dynamics, in particular Gottschalk's surjunctivity conjecture \cite{GromovSofic,KLi},
group rings and Kaplansky's direct finiteness conjecture \cite{ESZKaplansky} (also for linear sofic groups \cite[Prop 2.6]{LinearSofic}),
and $L^2$--invariants 
\cite{ElekSzaboDeterminant,Luck}.
Hyperlinear groups are of interest in operator algebras, particularly the Connes embedding theorem \cite{Connes}, 
and in group theory, particularly for the Kevare conjecture \cite[Cor 10.4]{Pestov}.
We refer the reader to  \cite{Pestov,CapLup} for surveys on sofic and hyperlinear groups.

There are many examples of sofic groups, including all amenable groups, all residually finite groups, and all linear groups (by Malcev's Theorem).
However, because of the weakness of the approximation by finite groups, few permanence properties of soficity are properly understood. 
Relatively straightforward examples include closure under direct product and increasing unions, and the soficity of residually sofic groups.
More substantial results generally require some amenability assumption.
For example, an amalgamated product of two sofic groups is know to be sofic if the amalgamated subgroup is amenable (see \cite{ESZ2,LPaun,DKP,PoppArg}). 
This was extended to encompass the fundamental groups of all graphs of groups with sofic vertex groups and amenable edge groups \cite{CHR}.
In the same paper, it is shown that the graph product of sofic groups is sofic.
Also, if $H$ is a subgroup of $G$ that is sofic and coamenable, then $G$ is sofic too \cite{ESZ1}.

 Our first result is a new permanence property for soficity that concerns wreath products.
 Recall that the wreath product of two groups $G$ and $H$ is the semidirect product $\bigoplus_H G \rtimes H$.
\begin{thmspecial}\label{thmspecial:sofic}
	Let $G,H$ be countable, discrete, sofic groups.
		Then $G\wr H$ is sofic.
\end{thmspecial}

When $G$ is abelian, Theorem \ref{thmspecial:sofic} was proved by Paunescu \cite{LPaun}, who used methods of analysis and the notion of sofic equivalence relations developed by Elek and Lippner \cite{ElekLip}.
By Elek and Szabo \cite[Theorem 1]{ESZ1} it follows that $G\wr H$ is sofic if $G$ is sofic and $H$ is amenable. From this, we may apply the work of Vershik and Gordon on local embeddability into finite groups to see that $G\wr H$ is sofic if $G$ is sofic and $H$ is locally embeddable into amenable groups (this follows from the proof of \cite[Proposition 3]{VerGor}). We refer the reader to Holt and Rees \cite{HR} for other results on metric approximation of wreath products (e.{}g.{} for commutator-contractive length functions).

We remark that, using the Magnus embedding (see \cite{Sale_Magnus} for both the original and a modern geometric definition), Theorem \ref{thmspecial:sofic} implies the following (in fact it follows from the weaker version of Paunescu, mentioned above \cite{LPaun}).

\begin{corspecial}
	Let $N$ be a normal subgroup of a finite rank free group $F$, and let $N'$ be the derived subgroup of $N$. If $F/N$ is sofic, then $F/N'$ is sofic.  
\end{corspecial}

\begin{proof}
	The Magnus embedding is $F/N' \hookrightarrow \Z^r \wr (F/N)$. 
	Since soficity passes to subgroups we therefore get the corollary from Theorem \ref{thmspecial:sofic}, or \cite{LPaun}.
\end{proof}

Sofic groups are also weakly sofic, linear sofic and hyperlinear, and  we ask to what extent these properties are preserved by wreath products.
Weakly sofic groups are a  class of groups which can be approximated by finite groups in a weaker sense than sofic groups, namely one is allowed to approximate $G$ by any finite group with any bi-invariant metric, instead of just permutation groups with the Hamming distance,  as is the case for soficity (see Section \ref{sec:apps} for precise definitions). 
Linear soficity and hyperlinearity are each classes of groups which can be approximated by linear groups---to be linear sofic requires approximation by general linear groups with respect to the rank metric, while hyperlinearity requires approximation by unitary groups in the normalized Hilbert-Schmidt distance. 

Our techniques proving Theorem \ref{thmspecial:sofic} generalize to give the following, broader result.

\begin{thmspecial}\label{thmspecial:sofic wreath}
	Let $G,H$ be countable, discrete groups and assume that $H$ is sofic. Then:
\begin{enumerate}[(i)]
\item If $G$ is sofic, then so is $G\wr H,$ \label{I:sofic}
\item If $G$ is hyperlinear, then so is $G\wr H,$\label{I:hyper}
\item If $G$ is linear sofic, then so is $G\wr H,$ \label{I:linear}
\item If $G$ is weakly sofic, then so is $G\wr H.$\label{I:weak}
\end{enumerate}
\end{thmspecial}

The proof of Theorem \ref{thmspecial:sofic wreath} is constructive,  and is almost entirely self-contained, the exceptions being the use of equivalent definitions of soficity, hyperlinearity, and linear soficity, and a result used for \eqref{I:linear} that concerns the behaviour of Jordan blocks under tensor products. 
The first step in the proof of Theorem \ref{thmspecial:sofic wreath} is a general result on metric approximations of groups, the proof of which is quantitative (see Proposition \ref{prop:metric approx}). 

Part \eqref{I:weak} follows immediately from Proposition \ref{prop:metric approx}, while the other three parts require extra constructions.

	We remark that the arguments in the matricial cases (\ref{I:hyper}),(\ref{I:linear}) are more delicate, as each of these arguments require tensor products of operators.
	In (\ref{I:linear}), for example, linear soficity of $G$ allows us to find  almost homomorphisms $\theta\colon G\to \GL_{n}(\FF)$, for some field $\FF$, so that $\frac{1}{n}\rank(\theta(g)-\id)$ is bounded away from zero for $g\in G\setminus\{1\}$. 
	However, this property is not stable under taking tensor products: for example if $\frac{1}{n}\rank(\theta(g)-5\id)$  and $\frac{1}{n}\rank(\theta(h)-\frac{1}{5}\id)$ 
	 are both small for some $g,h\in G,$ then $\frac{1}{n^{2}}\rank(\theta(g)\otimes \theta(h)-\id)$ will be small. 
	Because of this issue, we have to remark that linear soficity in fact implies that we can find an almost homomorphism $\theta\colon G\to \GL_{n}(\FF)$ so that $\inf_{\lambda\in \FF\setminus\{0\}}\frac{1}{n}\rank(\theta(g)-\lambda\id)$ is bounded away from zero for $g\in G\setminus\{1\}$. 
	A similar issue occurs in the hyperlinear case, where we find an almost homomorphism $\theta\colon G\to \mathcal{U}(n)$ so that $\theta(g)$ stays a bounded distance away from the scalar matrices (a result of Radulescu \cite{Radul} enables us to do this for case \eqref{I:hyper}).
	In each of these cases, forcing the image of our group elements to be far away from the scalars is a property that is stable under tensor products. This is a direct computation in the unitary case, whereas the argument that this is true in the general linear case is more involved (see Proposition \ref{prop:tesnorlowerbound}).

The structure of the paper is as follows. 
  Section \ref{sec:prelims} contains the definition of $\calC$--approximable groups, and the definition of sofic groups that we use. This section also looks at how we may determine that a map from a wreath product to a group is almost multiplicative, and how we endow our wreath products with suitable metrics.
  Once this is established, we give the initial construction of the metric approximations of a wreath product in Section \ref{S:construction}, before extending this to each of the specific cases of Theorem \ref{thmspecial:sofic wreath} in Section \ref{sec:apps}.

\textbf{Acknowledgments.} 
We wish to thank the diligence and thorough work of the anonymous referee.
The first-named author would like to thank Jesse Peterson for asking him if wreath products of sofic groups are sofic at the NCGOA Spring Institute in 2012 at Vanderbilt University.

\section{Preliminaries}\label{sec:prelims}

We begin with the necessary definitions, as well as a useful lemma to help us identify metric approximations in wreath products.
We first establish some notation.

\begin{notation}
Throughout we will use $1$ to denote the identity element of a group (we expect the reader to be able to infer which group it comes from), except when we talk of the identity matrix, when we use $\id$.
\end{notation}

The \emph{metric approximations},  to which we have referred, can be defined as an embedding of a group into a metric ultraproduct of groups, each with a given bi-invariant metric.
Such an embedding gives rise to a sequence of maps to the groups in the ultraproduct.
It is these maps on which we focus our attention.

Before we define key properties of these maps, we remind the reader that a metric $d$ on a group $H$ is said to be \emph{bi-invariant} if $d(axb,ayb)=d(x,y)$ for all $a,b,x,y\in H.$
Throughout the paper we will work with such metrics and their corresponding length functions.
We recall that a function $\ell\colon H\to [0,\infty)$ is a \emph{length function} on $H$ if:
	\begin{itemize}
		\item $\ell(h)=\ell(h^{-1})$ for all $h\in H,$
		\item $\ell(gh)\leq \ell(g)+\ell(h)$ for all $g,h\in H.$ 
	\end{itemize}
	We say that $\ell$ is \emph{conjugacy-invariant} if also $\ell(xgx^{-1})=\ell(g)$ for all $x,g\in H.$

A conjugacy-invariant length function $\ell$ on $G$ defines a bi-invariant metric by
$d(x,y)=\ell(y^{-1}x)$.
Conversely if $G$ has a bi-invariant metric $d,$ then $\ell(x)=d(x,1)$ is a conjugacy-invariant length function.

\begin{notation}
	When we switch between metrics and length functions we will pair them up with equivalent decorations on the notation.
	For example, a metric $d'$ will correspond to a length function $\ell'$.
\end{notation}

\begin{defn} 
	Let $H$ be a group with a bi-invariant metric $d.$ Fix a group $G$ and a function $\theta\colon G\to H.$ 
\begin{enumerate}[(a)]
	\item Given $F\subseteq G$ and $\varepsilon>0$ we say that $\theta$ is \emph{$(F,\varepsilon,d)$--multiplicative} if $\theta(1)=1$ and
	\[\max_{g,h\in F}d\big(\theta(gh),\theta(g)\theta(h)\big)<\varepsilon.\]
	
	\item Given $F\subseteq G$ and a function $c\colon G\setminus\{1\}\to (0,\infty)$ we say that $\theta$ is \emph{$(F,c,d)$--injective} if for all $g\in F\setminus \{1\}$
	$$d\big(\theta(g),1\big)\geq c(g).$$
\end{enumerate}	
\end{defn}

We remark that we will use the phrases \emph{almost multiplicative} and \emph{almost injective} to mean $(F,\varepsilon,d)$--multiplicative and $(F,c,d)$--injective respectively when we do not wish to specify $F,\varepsilon, c$ and $d$.

\begin{defn}
	Let $\mathcal{C}$ be a class of pairs $(H,d)$, where $H$ is a group and $d$ a bi-invariant metric on $H$ (the same group may appear multiple times in $\calC$ with different metrics). 
	We say that a group $G$ is \emph{$\mathcal{C}$--approximable} if there is a function $c\colon G\setminus\{1\}\to (0,\infty)$ so that for every finite $F\subseteq G$ and $\varepsilon>0$ there is a pair $(H,d)\in \mathcal{C}$ and an $(F,\varepsilon,d)$--multiplicative function $\theta\colon G\to H$ which is also $(F,c,d)$--injective.
\end{defn}

A special example of $\calC$--approximable groups are sofic groups, where $\calC$ consists of the finite symmetric groups paired with the normalized Hamming distance (see \cite{ElekSzaboDeterminant}).

\begin{defn}\label{def:Hamm} Let $A$ be a finite set. The \emph{normalized Hamming distance}, denoted $d_{\Hamm},$ on $\Sym(A)$ is defined by
\[d_{\Hamm}(\pi,\tau)=\frac{1}{\modulus{A}}\left|\{a\in A:\pi(a)\ne \tau(a)\}\right|.\]
The corresponding length function is denoted $\ell_{\Hamm}$.
\end{defn}

For our purposes, we will use an alternative (but equivalent) definition of soficity (see \cite[Thm. 1]{ElekSzaboDeterminant}).

\begin{defn} Let $G$ be a countable discrete group, $F$  a finite subset of $G$, and $\varepsilon>0.$ Fix a finite set $A$ and a function $\sigma\colon G\to \Sym(A).$ 
 We say that $\sigma$ is \emph{$(F,\varepsilon)$--free} if
\[\min_{g\in F\setminus\{1\}}\ell_{\Hamm}(\sigma(g))> 1-\varepsilon.\]
We say that $\sigma$ is an \emph{$(F,\varepsilon)$--sofic approximation} if it is $(F,\varepsilon,d_{\Hamm})$--multiplicative, and $(F,\varepsilon)$--free.
Lastly, we say that $G$ is \emph{sofic} if for every finite $F\subseteq G$ and $\varepsilon>0,$ there is a finite set $A$ and an $(F,\varepsilon)$--sofic approximation $\sigma\colon G\to \Sym(A).$
\end{defn}

Our aim is to  start with  approximations for $G$ and $H$ and use them to build  approximations for the wreath product $G\wr H$.  We will use permutational wreath products in our approximations, and we recall here the definitions. Note that all  our wreath products are of the restricted variety, meaning we  use direct sums rather than direct products.

\begin{defn}
Let $X$ be a set on which $H$ acts.
The \emph{permutational wreath product} is defined as 
		$$G\wr_X H = \bigoplus\limits_{X}G \rtimes H$$
where the action of $h\in H$ is given via $\alpha_{h}\in \Aut\left(\bigoplus_{X}G\right)$, defined by a coordinate shift:
		\[\alpha_{h}\Big( (g_{x})_{x\in X} \Big)=(g_{h\m x})_{x\in X}.\]
The regular wreath product $G\wr H$ is defined as above, taking $X=H$ with $H$ acting on itself by left-multiplication.
\end{defn}

A homomorphism $\varphi\colon G\wr H\to K$, for some group $K$, can be decomposed into a pair of homomorphisms $\varphi_{1}\colon \bigoplus_{H}G\to K$, $\varphi_{2}\colon H\to K$ which satisfy the following equivariance condition:
\[\varphi_{2}(h)\varphi_{1}(g)=\varphi_{1}(\alpha_{h}(g))\varphi_{2}(h),\mbox{ for all $h\in H,g\in \bigoplus_{H}G.$}\]
The following lemma gives an analogue to this for the case of almost multiplicative maps.

\begin{lem}\label{l:almosthomom} Let $G,H$ be countable, discrete groups, 
	and let
	$$\proj_H \colon G \wr H \to H \textrm{ and } \proj_G \colon G\wr H \to \bigoplus_H G$$ 
	be the natural projection maps  (note that the latter is not a homomorphism).
	For a finite subset $F_0\subseteq G\wr H$ define subsets 
	\begin{align*}
	E_1 &= \big\{ \alpha_h (g) : h\in \proj_H(F_0) \cup \{1\} , g\in \proj_G(F_0) \big\}, \\
	E_2 &= \proj_H(F_0).
	\end{align*}
	
	Let $\varepsilon>0$ and $K$ be a group with a bi-invariant metric $d$. Suppose $\Theta\colon G\wr H\to K$ is a map with $\Theta(1)=1$ such that 
	
	\begin{itemize}
		\item the restriction of $\Theta$ to $\displaystyle \bigoplus_{H}G$ is $(E_{1},\varepsilon/6,d)$--multiplicative,
		\item the restriction of $\Theta$ to $H$ is $(E_{2},\varepsilon/6,d)$--multiplicative, \phantom{$\displaystyle \bigoplus_{H}G$} 
		\item $\displaystyle \max_{g\in E_{1},h\in E_{2}}d\big(\Theta(g,h),\Theta(g,1)\Theta(1,h)\big)<\varepsilon/6$,
		\phantom{$\displaystyle \bigoplus_{H}G$} 
		\item $\displaystyle \max_{g\in E_{1},h\in E_{2}}d\left(\Theta(1,h)\Theta(g,1),\Theta(\alpha_{h}(g),1)\Theta(1,h)\right)<\varepsilon/6$.
		\phantom{$\displaystyle \bigoplus_{H}G$} 
	\end{itemize}
	Then $\Theta$ is $(F_0,\varepsilon,d)$--multiplicative.
	
\end{lem}

\begin{proof}
	Note that if $(g,h),(g',h')$ are in $F_0$, then  $g, g',\alpha_h(g') \in E_1$, and $h,h'\in E_2$.
	Applying the triangle inequality gives the result.
	Verification of this is left to the reader.
\end{proof}

In our construction we use maps to groups of the form $L \wr_B \Sym(B)$, for some group $L$ endowed with a bi-invariant metric.
To make sense of the notions of almost multiplicativity and almost injectivity we need a bi-invariant metric on this wreath product.
The following proposition explains how we do this, using the language of length functions.
This was described independently by Holt and Rees \cite[\S 5]{HR}.

\begin{prop}\label{P:conjinvlengthfunc} Let $L$ be a group with a conjugacy-invariant length function $\ell$ and suppose that $\ell(g)\leq 1$ for all $g\in G$.
For a finite set $B$, define $\tilde\ell$ on $L\wr_{B}\Sym(B)$ by
\[\tilde\ell\big((k_{b})_{b\in B},\tau\big)=\ell_{\Hamm}(\tau)+\frac{1}{|B|}\sum_{\substack{b\in B\\\tau(b)=b}}\ell(k_{b}).\]
Then $\tilde\ell$ is a conjugacy-invariant length function.
\end{prop}

\begin{proof}
We first show that $\tilde\ell$ is conjugacy-invariant. 
Fix $h=(h_b),k=(k_b)\in \bigoplus_{B}L$ and $\pi,\tau\in \Sym(B).$ 
Then
\[(k,\tau)^{-1}(h,\pi)(k,\tau) = \big( \alpha_{\tau\m}(k^{-1}h)  \alpha_{\tau\m\pi}(k) , \tau^{-1}\pi\tau \big).\]
Using the conjugacy-invariance of $\ell_{\Hamm}$ we have:
\[\tilde\ell\big((k,\tau)^{-1}(h,\pi)(k,\tau)\big) 
=
 \ell_{\Hamm}(\pi) 
 +
 \frac{1}{|B|}\sum_{\substack{b\in B\\\tau^{-1}\pi\tau(b)=b}} \ell(k^{-1}_{\tau(b)}h_{\tau(b)}k_{\pi^{-1}\tau(b)}).\]
Note that if $\tau^{-1}\pi\tau(b)=b,$ then $\tau(b)=\pi^{-1}\tau(b)$.
We can use this to rewrite the summation term above, and then use the conjugacy invariance of $\ell$ to further simplify it:
\[\frac{1}{|B|} \sum_{\substack{b\in B\\\pi\tau(b)=\tau(b)}} \ell(k^{-1}_{\tau(b)}h_{\tau(b)}k_{\tau(b)}) 
=
\frac{1}{|B|} \sum_{\substack{b\in B\\\pi\tau(b)=\tau(b)}} \ell(h_{\tau(b)}) 
=
\frac{1}{|B|}\sum_{\substack{b\in B\\\pi(b)=b}}\ell(h_{b}).\]
Thus we see that 
\[\tilde\ell\big((k,\tau)^{-1}(h,\pi)(k,\tau)\big) 
=
\ell_{\Hamm}(\pi)
+
\frac{1}{|B|} \sum_{\substack{b\in B\\ \pi(b)=b}} \ell(h_{b})=\tilde\ell(h,\pi).\]
The proof that $\tilde\ell\big((k,\pi)^{-1}\big)=\tilde\ell(k,\pi)$ is similar.

We now prove the triangle inequality. Take $h,k,\pi,\tau$ as above. Then
\begin{align*}
\ell \big( (k,\tau)(h,\pi) \big)
&=\ell_{\Hamm}(\tau\pi) 
+
\frac{1}{|B|} \sum_{\substack{b\in B\\\pi(b)=\tau^{-1}(b)}} \ell(k_{b}h_{\tau^{-1}(b)})\\
&=\ell_{\Hamm}(\tau\pi)
+
\frac{1}{|B|} \sum_{\substack{b\in B\\\pi(b)=\tau^{-1}(b)}} \ell(k_{b}h_{\pi(b)})
\end{align*}
Let $\hat{B} = \{b\in B:\pi(b)=\tau^{-1}(b) \ne b\}$. Then, using the fact that $\ell$ is bounded by $1$ we get
	$$\ell \big( (k,\tau)(h,\pi) \big) 
	\leq
	\ell_{\Hamm}(\tau\pi)
	+
	\frac{1}{\modulus{B}}\left( |{\hat{B}}| + \sum_{\substack{b\in {B}\\\tau(b)=b}} \ell(k_{b})+ \sum_{\substack{b\in B \\ \pi(b) = b}} \ell(h_{b}) \right).$$
From the above we see that it is enough to show that 
\[\ell_{\Hamm}(\tau\pi)+\frac{|\hat{B}|}{\modulus{B}}\leq \ell_{\Hamm}(\tau)+\ell_{\Hamm}(\pi).\]
Using the definition of the Hamming distance, we need
\[|\{b:\pi(b)\ne\tau^{-1}(b)\}|+|\hat{B}|\leq |\{b\in B:\tau(b)\ne b\}|+|\{b\in B:\pi(b)\ne b\}|.\]
Since $\hat{B} \subseteq \{b\in B : \pi(b) = b\}$, we get that the above is the same as:
\[|\{b:\pi(b)\ne \tau^{-1}(b)\}|\leq |\{b\in B:\tau^{-1}(b)\ne b\}| + |\{b\in B:\pi(b)\ne b,\pi(b)\ne \tau^{-1}(b)\}|,\]
which we can deduce from the inclusion
\[\{b:\pi(b)\ne \tau^{-1}(b)\}\subseteq \{b\in B:\tau^{-1}(b)\ne b\} \cup \{b\in B:\pi(b)\ne b,\pi(b)\ne \tau^{-1}(b)\}.\]
This completes the proof of the triangle inequality and thus of Proposition \ref{P:conjinvlengthfunc}.
\end{proof}

\section{Construction of the Approximation}\label{S:construction}

 In the following we let $G,H,K$ be groups, $B$ a finite set, and we suppose that functions $\theta\colon G\to K$ and $\sigma\colon H\to \Sym(B)$ (not necessarily homomorphisms) are given.

\subsection{Some intuition}\label{sec:intuition}
To give some idea of the intuition behind the construction that follows, consider first how one can think of an element of a wreath product $G\wr H$.
One may consider $(g,h) \in G\wr H$, where $g=(g_x)_{x\in H}$, as a journey through $H$, starting at the identity, finishing at $h$, and picking up elements of $G$ at selected points of $H$ (namely, pick up $g_x$ at $x$ whenever $g_x \neq 1$).

If both $G,H$ are sofic, we wish to construct a finite model for $G\wr H$ using symmetric groups (here $K= \Sym(A)$).
A sofic approximation, roughly speaking, gives us a finite set ($A$ or $B$), inside of which a significant part of the set behaves like a prescribed finite subset of $G$ or $H$ respectively (see e.g. \cite[p. 157]{GromovSofic}, \cite[Prop 4.4]{ESZKaplansky}).
We ultimately seek such a set for $G\wr H$, and first we may try to combine $A$ and $B$ in a way which mimics the wreath product of groups.
However this approach leads to a problem.

The problem is that in the approximation of $H$ using $B$ there is no prescribed point in $B$ representing the identity.
Thus the ``journey'' through $H$ from the identity to $h$ will translate to a ``journey'' in $B$ from $\beta$ to $b=\sigma(h)\beta$, where the choice of $\beta$ is arbitrary, and may be allowed to vary.

It is for this reason that, in the construction below, we use $\bigoplus_B K$, rather than just $K$, in the  wreath product $(\bigoplus_B K )\wr_B \Sym(B)$ that we map into. 
This could be interpreted as using one copy of $K$ for each choice of ``identity vertex'' in $B$.

\subsection{The construction}\label{sec:construction subsec}
The aim of this section is to define an approximation of $G\wr H$ into a wreath product that is, in some sense, smaller,  or more controllable, than $G\wr H$. 
In Section \ref{sec:apps} this approximation is then used to prove each part of Theorem \ref{thmspecial:sofic wreath}, composing it with a further approximation into the specific type of group for each property.
The only exception is part \eqref{I:weak}, where weak soficity follows immediately from this construction and Proposition \ref{prop:metric approx}.
	Given a group $K$ and a set $B$, we consider the wreath product
	\[\left(\bigoplus_B K\right) \wr_B \Sym(B) =   \bigoplus_{B}\left(\bigoplus_{B}K\right)\rtimes \Sym(B)  ,\]
	so $\pi \in \Sym(B)$ acts on $\bigoplus_{B}\left(\bigoplus_{B}K\right)$ by $\alpha_{\pi}$ where
	\[\alpha_{\pi}\big((k_{b})_{b\in B}\big)=(k_{\pi^{-1}(b)})_{b\in B},\mbox{ if $k_{b}\in \bigoplus_{B}K$ for all $b\in B$.}\]
	Note that if we identify $k\in \bigoplus_{B}\left(\bigoplus_{B}K\right)$ with an element $(k_{b,\beta})_{b,\beta\in B}$ in $\bigoplus_{B\oplus B}K,$ then
	\[\alpha_{\pi}((k_{b,\beta}))=(k_{\pi^{-1}(b),\beta}).\]

\begin{notation} \label{not:romans vs greeks}
	When we encounter sets of the form $\bigoplus_B (\bigoplus_B A)$, we will use roman subscripts to identify the outer index, and greek subscripts to identify the inner index.
	For example, for $a = (a_b)_{b\in B}$, we have each $a_b \in \bigoplus_B A$, expressed as $a_b = (a_{b,\beta})_{\beta \in B}$.
\end{notation}

Given the maps $\theta \colon G \to K$ and $\sigma \colon H \to \Sym(B)$
 and a finite subset $E$ of $H$
 we define 
\[\Theta : G\wr H \to \bigoplus_B K \wr_B \Sym(B)\]
by
$\Theta(g,h)=(\theta_{B}(g),\sigma(h))$,
where $\theta_B$ is a map we proceed to define below.

We use the finite subset $E\subset H$ to define a subset $B_E$ of $B$, given as the intersection $B=B_1\cap B_2$, where 
\[B_{1}=\{b\in B:\sigma(h_{1})\m b\ne \sigma(h_{2})\m b\mbox{ for all $h_{1},h_{2}\in E,h_{1}\ne h_{2}$}\},\]
\[B_{2}=\{b\in B:\sigma(h_{1}h_{2})\m b=\sigma(h_{2})\m\sigma(h_{1})\m b\mbox{ for all $h_{1},h_{2}\in E$}\}.\]

If we consider an element of $\bigoplus_B(\bigoplus_B K)$ as a function $B\times B \to K$, where we follow the convention of Notation \ref{not:romans vs greeks}, then $\theta_B$ is the map defined by
\[(g_h)_{h\in H} \mapsto \begin{cases}
(b,\beta) \mapsto \theta(g_{h_0}), & \textrm{if $b\in B_{E},$ if $g_{h}=1$ for all $h \in H\setminus E,$} \\ & \textrm{
	and $h_{0}\in H$ is so that $b=\sigma(h_0)\beta$;} \\
\vspace{-3mm}&\\
(b,\beta)\mapsto 1, & \textrm{otherwise.}
\end{cases}\]

Referring back to the intuition of Section \ref{sec:intuition}, we explain what happens when we fix $\beta$, the second coordinate in $B\times B$.
This coordinate is the inner index for an element of $\bigoplus_B (\bigoplus_B K)$, and corresponds to a choice of a copy of $K$ in the wreath product $\bigoplus_B K \wr_B \Sym(B)$.
If $\beta\in \bigcap_{x\in E}\sigma(x)^{-1}B_{E}$,
then $\theta_B$ restricts to a map that sends $(g_x)_{x\in E}$ to an element of $\bigoplus_B K$, where for each $x\in E$, the $\sigma(x)\beta$--coordinate is given by $\theta(g_x)$, and coordinates not of the form $\sigma(x)\beta$ for $x\in E$ are trivial.
Thus, we can think of $\beta$ as behaving as the chosen ``identity vertex'' in $B$.
If an element $(g,h)$ of $G\wr H$ is a journey through $H$, starting at $1$ and picking up elements of $G$ en route to $h$,
then under $\Theta$ this turns into a collection of journeys through $B$, each starting at a suitable choice of $\beta$, and finishing at $\sigma(h)\beta$.  
The map $\theta_B$ tells you what elements of $K$ to pick up along the way.
If the original journey visited a vertex $x\in E$, then the image journey starting at $\beta$ will pick up $\theta(g_x)$ at $\sigma(x)\beta$.
This is visualized in Figure \ref{fig:intuition}

\begin{figure}[h!]
	\labellist
	\footnotesize
	\pinlabel $H$ at 120 300
	\pinlabel $E$ at 160 220
	\pinlabel $1$ at 122 178
	\pinlabel $h$ at 122 228
	\pinlabel $x$ at 98 120
	\pinlabel $B$ at 415 210
	\pinlabel $\beta$ at 384 178
	\pinlabel $\sigma(h)\beta$ at 384 228
	\pinlabel $\sigma(x)\beta$ at 355 113
	\endlabellist
	\includegraphics[width = 12cm]{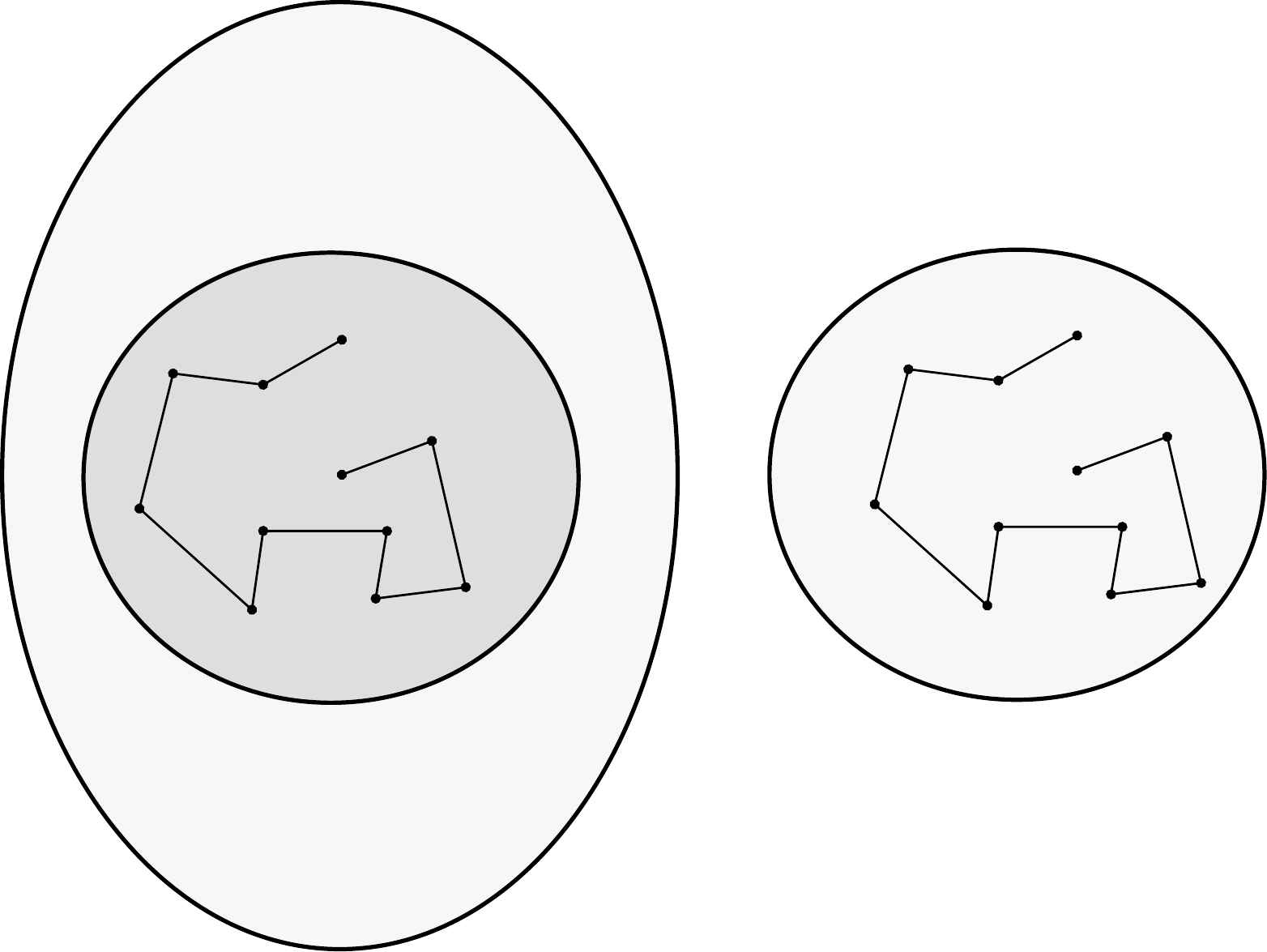}
	\caption{\label{fig:intuition} The journey through $H$ corresponding to $(g,h)$ on the left, when $g\in \bigoplus_E H$, and the journey in the image of $\Theta$ corresponding to choosing $\beta$ to play the role of the identity. The image $\Theta(g,h)$ will be made up of multiple such journeys, one for each suitable choice of $\beta$.}
\end{figure}

We now give an equivalent definition of $\theta_B$.
This is necessary in order to establish further maps and notation which will be used later on.
For $h\in H,b\in B,$ define
\[\theta_{b}^{(h)}\colon G\to \bigoplus_{B}K\]
by $\theta_{b}^{(h)}(g)=(k_{\beta})_{\beta\in B}$ where
\[k_{\beta}=\begin{cases}
\theta(g),&\textnormal{ if $\beta=\sigma(h)\m b$},\\
1,&\textnormal{ otherwise.}
\end{cases}\]

Note that $\theta_{b}^{(h_{1})}(g_{1})$ and $\theta_{b}^{(h_{2})}(g_{2})$ 
commute if $b\in B_1,h_{1},h_{2}\in E,g_{1},g_{2}\in G$ and $h_{1}\ne h_{2}.$ 
Thus it makes sense to define, for $b\in B_E,$ 
\[\theta_{b}\colon \bigoplus_{E}G\to \bigoplus_{B}K\]
by
\[\theta_{b}\Big((g_{h})_{h\in E}\Big)=\prod_{h\in E}\theta^{(h)}_{b}(g_{h}).\]
In our applications $\sigma$ will be a sofic approximation, so we can think of $B_E$ as making up the majority of $B$. 
Thus $\theta_{b}$ will be defined for ``most" $b\in B.$  
We extend $\theta_{b}$ to be defined for all $b\in B$ by saying that 
$\theta_{b}$ maps everything to the identity for $b\in B\setminus B_E$. 

Relating this to the intuition described above, for $g\in \bigoplus_E G$ and $\beta\in B$, the $\beta$--coordinate of $\theta_b(g)$ tells  you what element of $K$ to pick up at $b$ if $\beta$ is chosen as the ``identity vertex.''

We then obtain our equivalent definition of $\theta_B\colon \bigoplus_{E}G\to \bigoplus_{B}\left(\bigoplus_{B}K\right)$
by packaging all these maps together as a single map
$$\theta_{B}(g)=\big(\theta_{b}(g)\big)_{b\in B}$$
and extending $\theta_{B}$ to $\bigoplus_{H}G$ by declaring that $\theta_{B}(g)=1$ if $g\in \bigoplus_{H}G,$ but $g\notin \bigoplus_{E}G.$

\medskip

We will prove that  if $K$ has a bi-invariant metric and $\theta$ and $\sigma$ are almost multiplicative and almost injective, then $\Theta$ gives us our desired almost multiplicative and almost injective map.
To do this, we need to use an appropriate bi-invariant length function on $\left(\bigoplus_{B}K\right)\wr_B \Sym(B).$ 

\begin{notation} \label{notn:length f unctions}
Let $\ell'$ be a conjugacy-invariant length function on $\bigoplus_B K$ such that $\ell'\leq 1$.
Take $\tilde\ell$ to be the conjugacy-invariant length function on $\left(\bigoplus_{B}K\right)\wr_B \Sym(B)$ defined in Proposition \ref{P:conjinvlengthfunc}. 
When $\ell'$ is the specific length function $\ell_{\max}$ defined on $\bigoplus_{B}K$ by
\[\ell_{\max}\big((k_{b})_{b\in B}\big)=\max_{b\in B}\ell(k_{b}),\]
where $\ell$ is a given length function on $K$ such that $\ell\leq 1$,
we denote the length function obtained from Proposition \ref{P:conjinvlengthfunc} by $\tilde \ell_{\max}$, to emphasise the specific choice of $\ell'$.

In summary, we will be dealing with the following length functions, with corresponding metrics:
\begin{itemize} 
	\item  $\ell$ on $K$, corresponding to $d$, and such that $\ell\leq 1$;
	\item $\ell'$ on  $\bigoplus_B K$, corresponding to $d'$, and such that $\ell'\leq 1$;
	\item $\ell_{\max}$ on $\bigoplus_B K$---we never refer to the corresponding metric;
	\item $\tilde \ell$ on $\left(\bigoplus_{B}K\right)\wr_B \Sym(B)$, corresponding to $\tilde d$;
	\item $\tilde{\ell}_{\max}$ on $\left(\bigoplus_{B}K\right)\wr_B \Sym(B)$, corresponding to $\tilde d_{\max}$.
\end{itemize}
\end{notation}

Our aim is to prove the following.

\begin{prop}\label{prop:metric approx}
	Let $F\subseteq G\wr H$ be finite and $\varepsilon>0.$ 
There are finite subsets $E_G\subseteq G$ and $E,E_H\subseteq H$, and  an $\varepsilon'>0$ with the following properties.
Let
\begin{itemize} 
	\item $\sigma\colon H\to \Sym(B)$ be an $(E_H,\varepsilon')$--sofic approximation,
  	\item $\theta : G \to K$ be a map,
  	\item $\ell$, $d$, $\ell'$, $d'$, $\tilde \ell$, $\tilde d$, $\tilde \ell_{\max}$, and $\tilde d_{\max}$ be as described in Notation \ref{notn:length f unctions}.
\end{itemize}
  	Then $\Theta : G\wr H \to \left(\bigoplus_B K\right) \wr_B \Sym(B)$, as constructed above using $E$, $\theta$ and $\sigma$, has the following properties. 
\begin{enumerate}[(a)]
\item Suppose the length function $\ell'$ on $\bigoplus_B K$  restricts to $\ell$ on each copy of $K$.

If $\theta\colon G\to K$ is $(E_G,\varepsilon',d)$--multiplicative,
then $\Theta$ is $(F,\varepsilon,\tilde d)$--multiplicative.

\item
Let  $c$ be a map $c\colon G\setminus\{1\}\to (0,\infty)$. Define  $c'\colon (G\wr H)\setminus \{1\}\to (0,\infty)$ by
\[c'(g,h)=\begin{cases}
\frac{1}{2}, &\textnormal{ if $h\ne 1$}\\
\max\limits_{x\in \supp(g)}\frac{1}{2}c({g_{x}}), &\textnormal{ if $h=1,$ $g=(g_{x})_{x\in H}.$}\\
\end{cases}\]
Then, if  $\theta\colon G\to K$ is $(E_G,c,d)$--injective
then $\Theta$ is $(F,c',\tilde d_{\max})$--injective.
\end{enumerate}
\end{prop}

The remainder of this section is dedicating to proving Proposition \ref{prop:metric approx}.
We will see below that, once $E$ is given, the following upper bounds on $\varepsilon'$ are sufficient:
\begin{equation}\label{eq:epsilon' bounds}\begin{array}{ll} \textrm{for (a):} & \varepsilon' < \frac{\varepsilon}{48|E|^{2}},   \\ 
\textrm{for (b):} & \varepsilon' < \frac{1}{16\modulus{E}^2}\min\big\{{c(g)},{1} \mid g\in E_G\setminus\{1\} \big\}. \end{array}\end{equation}
As we see, the bounds on $\varepsilon'$ depend only on $\varepsilon$ and the set $F$.

{We remark that $\Theta(1,1)=1$ by construction.}
We first explain how to define the sets $E$, $E_G$ and $E_H$.

Let $F\subseteq G\wr H$ be finite and $\varepsilon>0$.
Define projections $\proj_{G}\colon G\wr H\to \bigoplus_{H}G$ and $\proj_{H}\colon G\wr H\to H$ by $\proj_{G}(g,h)=g$ and $\proj_{H}(g,h)=h$.
Let $E_{1},E_{2}$ be as in Lemma \ref{l:almosthomom} for the finite set $F_0=F\cup \{1\} \cup F\m$:
		\[E_1 = \big\{ \alpha_h (g) : h\in \proj_H(F_0) , g\in \proj_G(F_0) \big\} \subseteq \bigoplus_E G, \ \ 
		E_2 = \proj_H(F_0) \subseteq H.\]
Recall that for $g=(g_{x})_{x\in H}\in \bigoplus_{H}G$ the support of $g,$ denoted $\supp(g),$ is the set of $x\in H$ with $g_{x}\ne 1.$  We set
\[E = E_2 \cup \bigcup_{\substack{g\in E_1\\h\in E_2}}h\supp(g),\ \ 
E_G = \big\{ g_x \in G : (g_x) \in E_1, x\in H \big\},\ \  
E_H=E\m E .\]
Since $E_2$ contains the identity, it follows that $E$ and $E^{-1}$ are both subsets of $E_H$.

Let $K$ be as in Proposition \ref{prop:metric approx}  
and let $\theta\colon G\to K$ be $(E_{G},\varepsilon',d)$--multiplicative and $(E_{G},c,d)$--injective,  where $\varepsilon'$ is controlled by the bounds in \eqref{eq:epsilon' bounds} above.
Let $ \sigma\colon H\to \Sym(B)$ be a $(E_H,\varepsilon')$--sofic approximation.
Recall the set $B_E$ is defined from $E$ as the intersections of sets $B_1,B_2$ (which depend only on $E$).
Lemma \ref{lem:prelim obs} confirms that, since $\sigma$ is a sofic approximation, $B_E$ makes up a significant proportion of the set $B$.

\begin{lem}   \label{lem:prelim obs}
	Let $\kappa>0$.  If $ \varepsilon' < \frac{\kappa}{4\modulus{E}^2}$ then  $\modulus{B\setminus B_E}\leq \kappa|B|.$
\end{lem}

\begin{proof}
	Note that
\[B\setminus B_{1}=\bigcup_{\substack{h_{1},h_{2}\in E\\h_1\ne h_2}}\{b\in B:\sigma(h_{1})\m b= \sigma(h_{2})\m b\}.\]
Since $E_{H}\supseteq E\cup E^{-1},$ by $(E_{H},\varepsilon')$--soficity of $\sigma$ we have $d_{\Hamm}(\sigma(h_{2})^{-1},\sigma(h_{2}^{-1}))<\varepsilon'.$ 
Thus for $h_{1}\ne h_{2},$ we have
\begin{eqnarray*}\frac{\modulus{\{b\in B:\sigma(h_{1})\m b=\sigma(h_{2})\m b\}}}{\modulus{B}}&=&d_{\Hamm}(\sigma(h_{1})^{-1},\sigma(h_{2})^{-1})\\
&=& 1-\ell_{\Hamm}(\sigma(h_{1})\sigma(h_{2})^{-1})\\
&< & 1-\ell_{\Hamm}(\sigma(h_{1})\sigma(h_{2}^{-1}))+\varepsilon'\\
&\leq & 1- \ell_{\Hamm}(\sigma(h_{1}h_{2}^{-1}))+2\varepsilon'\\
&< & 3\varepsilon',
\end{eqnarray*}
where in the last two lines we again use that  $E_H\supseteq E\cup E\m\cup E\m E$.  Thus
\[\frac{\modulus{B\setminus B_{1}}}{\modulus{B}}\leq 3\modulus{E}^{2}\varepsilon'.\]
Similarly, $(E_H,\varepsilon',d_{\Hamm})$--multiplicativity of $\sigma$ gives
\[\frac{\modulus{B\setminus B_{2}}}{\modulus{B}}\leq \sum_{h_{1},h_{2}\in E}\Big(1-d_{\Hamm}\big(\sigma(h_{1}h_{2}),\sigma(h_{1})\sigma(h_{2})\big)\Big)\leq \modulus{E}^{2}\varepsilon'.\]
This proves the lemma.
\end{proof}

Use the set  $E \subset H$ and the maps $\theta,\sigma$ to define the maps $\theta_{B},\Theta$, as constructed at the start of this section. 

\subsection{Part (a) of Proposition \ref{prop:metric approx}}

We claim that if $\varepsilon'$ is sufficiently small, then the map $\Theta$ is $(F,\varepsilon,\tilde d)$--multiplicative.

Take $\kappa>0$ so that $\kappa < \frac{\varepsilon}{12}$, and take  $\varepsilon'>0$ satisfying the hypothesis of Lemma \ref{lem:prelim obs}, so we will  have $\varepsilon' <\frac{\varepsilon}{48|E|^{2}}$.

We now apply Lemma \ref{l:almosthomom}, verifying below the four necessary conditions to show that $\Theta$ is $(F,\varepsilon,\tilde d)$--multiplicative.
 We first check that it is $(E_1,\varepsilon/6,\tilde d)$--multiplicative when restricted to $\bigoplus_H G$. 
Recall that throughout Proposition \ref{prop:metric approx} we assume that $\ell\leq 1$ and $\ell'\leq 1$, while in part (a) we assume furthermore that $\ell'$ restricts to $\ell$ on each copy of $K$. Let $g,g'\in E_{1}$ with $g=(g_{x})_{x\in H},g'=(g'_{x})_{x\in H}$.
Since $E_1\subset \bigoplus_E G$, we may apply $\theta_b$ to $g$, $g'$, and $gg'$.
 Then
\begin{align*}
\tilde{d}(\theta_{B}(g)\theta_{B}(g'),\theta_{B}(gg'))&=\frac{1}{|B|}\sum_{b\in B}d'\big( \theta_{b}(g)\theta_{b}(g'), \theta_{b}(gg')\big)\\
&\leq \kappa+\frac{1}{|B|}\sum_{b\in B_E}d'\big(\theta_{b}(g)\theta_{b}(g'), \theta_{b}(gg')\big).
\end{align*}

By the definitions of $\theta_b$ and of $E_1$, we realise that each component of
$\theta_{b}(gg')\m \theta_{b}(g)\theta_{b}(g')$ is either $1$ or
$\theta(g_{x}g_{x}')^{-1}\theta(g_{x})\theta(g_{x}')$, for $x\in E$. 
Thus 
\begin{align*}
\tilde{d}(\theta_{B}(g)\theta_{B}(g'),\theta_{B}(gg'))&\leq \kappa+\frac{1}{|B|}\sum_{b\in B_E}\sum_{x\in E}d\big(\theta(g_{x})\theta(g_{x}'), \theta(g_{x}g_{x}')\big)\\
&\leq \kappa+\modulus{E}\varepsilon',
\end{align*}
where in the last line we use that $\theta$ is $(E_G,\varepsilon',d)$--multiplicative. Since $\kappa+\modulus{E}\varepsilon'<\frac{\varepsilon}{6},$ we see that $\Theta$ is $(E_{1},\varepsilon/6,\tilde d)$--multiplicative.

The fact that the restriction to $H$ is $(E_2,\varepsilon/6,\tilde d)$--multiplicative is more straightforward.
Indeed, for $h,h'\in E_{2}$ we have
\begin{equation*}
\tilde{d}(\sigma(hh'),\sigma(h)\sigma(h'))=d_{\Hamm}(\sigma(hh'),\sigma(h)\sigma(h'))<\varepsilon',
\end{equation*}
where we note that we can use the multiplicative property of $\sigma$ since $E_2 \subseteq E_H$.

 By construction, the third condition of Lemma \ref{l:almosthomom}, bounding the distance between $\Theta(g,h)$ and $\Theta(g,1)\Theta(1,h)$, is automatically satisfied by $\Theta$, since these elements are equal.

We finish part (a) by verifying the bound on
$\tilde{d}\big(\Theta(1,h)\Theta(g,1) , \Theta(\alpha_h(g),1)\Theta(1,h)\big)$  
for $g\in E_{1},h\in E_{2}.$ We have
\begin{align*}
\tilde{d}\big(\Theta(1,h)\Theta(g,1) , \Theta(\alpha_h(g),1)\Theta(1,h)\big)
&=\tilde{d}\big((\alpha_{\sigma(h)}(\theta_{B}(g)),\sigma(h)),(\theta_{B}(\alpha_{h}(g)),\sigma(h))\big)\\
&=\frac{1}{|B|}\sum_{b\in B}d'\big( \theta_{\sigma(h)^{-1}b}(g),  \theta_{b}(\alpha_{h}(g)\big)\\
&=\frac{1}{|B|}\sum_{b\in B}d'\big( \theta_{b}(g), \theta_{\sigma(h)b}(\alpha_{h}(g))\big).
\end{align*}
 Using Lemma \ref{lem:prelim obs}, and that $\ell' \leq 1$, we can disregard what happens for $b$ outside of both $B_E$ and $\sigma(h)\m B_E$ for a controlled cost. This gives us the following upper bound for the above distance:
$$ 2\kappa+\frac{1}{|B|}\sum_{b\in B_E\cap \sigma(h)\m B_E}d'\big(\theta_{b}(g), \theta_{\sigma(h)b}(\alpha_{h}(g))\big).$$
Since $\supp(\alpha_{h}(g))=h\supp(g)$, and $E$ contains both $\Supp(g)$ and $h\supp(g)$, it follows that for every $b\in B_E\cap \sigma(h)^{-1}(B_E)$ we have
\[\theta_{\sigma(h)b}(\alpha_{h}(g))=\prod_{x\in h\supp(g)}\theta_{\sigma(h)b}^{(x)}(g_{h^{-1}x})=\prod_{x\in \supp(g)}\theta_{\sigma(h)b}^{(hx)}(g_{x}).\]
Note that we have used that $\theta(1) = 1$ to restrict the number of terms in the product.
We use that for $h\in E$ (and hence for $h\in E_2$) and $b\in B_E\cap \sigma(h)\m B_E$ we have that $\theta_{\sigma(h)b}^{(hx)}(g)=\theta_{b}^{(x)}(g)$.
Inserting this into the above equation we see that
\[ \theta_{\sigma(h)b}(\alpha_{h}(g))= \prod_{x\in \supp(g)}\theta_{b}^{(x)}(g_{x})= \theta_{b}(g).\]
Returning to the above inequality, we have shown that
\[\frac{1}{|B|}\sum_{b\in B_E\cap \sigma(h)\m B_E}d'(\theta_{b}(g), \theta_{\sigma(h)b}(\alpha_{h}(g)))=0\]
so
\begin{equation*}
\tilde{d}\big(\Theta(1,h)\Theta(g,1) , \Theta(\alpha_h(g),1)\Theta(1,h)\big)
<2\kappa<\frac{\varepsilon}{6}.
\end{equation*}
This completes the proof of part (a) of Proposition \ref{prop:metric approx}.

\subsection{Part (b) of Proposition \ref{prop:metric approx}}

We now show that $\Theta$ is $(F,c',\tilde d_{\max})$--injective, when the length function $\ell'$ on $\bigoplus_B K $ is $\ell_{\max}$.

In order to get (b) we will need to further restrict the size of $\kappa$ (and hence also of $\varepsilon'$). We take $\kappa$ small enough so that, in addition to having $\kappa < \frac{\varepsilon}{12}$, we also have
		$$\kappa<\frac{1}{4}\min\big\{{c(g)},{1} \mid g\in E_G\setminus\{1\} \big\}.$$

First suppose $(g,h)\in F.$ If $h\ne 1,$ then
\[\tilde{\ell}_{\max}\big((\theta_{B}(g),\sigma(h))\big)\geq \ell_{\Hamm}(\sigma(h))\geq 1-\varepsilon'\geq 1/2=c'(g,h).\]\
We may therefore assume that $h=1.$
Let $g=(g_{x})_{x\in E}$. We then have that
\begin{align*}
\tilde{\ell}_{\max}\big((\theta_{B}(g),1)\big)&=\frac{1}{\modulus{B}}\sum_{b\in B}\ell_{\max}(\theta_{b}(g))\\
&\geq -\kappa+\frac{1}{\modulus{B}}\sum_{b\in B_E}\ell_{\max}(\theta_{b}(g))\end{align*}
using Lemma \ref{lem:prelim obs} to obtain the inequality.
Since for $b\in B_E$ the components of $\theta_b(g)$ are either $1$ or $\theta(g_x)$ for some $x \in E$, we get $\ell_{\max}(\theta_{b}(g)) = \max_{x\in E}\ell(\theta(g_x))$. Hence
\begin{align*}
\tilde{\ell}_{\max}\big((\theta_{B}(g),1)\big)
&\geq -\kappa+\frac{\modulus{B_E}}{\modulus{B}}\max_{x\in E}\ell(\theta(g_{x}))\\
&\geq -\kappa + (1-\kappa) \max_{x\in \Supp(g)}c({g_{x}})
\end{align*}
where the last inequality follows from Lemma \ref{lem:prelim obs} and the fact that $\theta$ is $(E_G,c,d)$--injective.
By the choices of $\kappa$ and $E_G$, we get 
$$-\kappa + (1-\kappa) \max_{x\in \Supp(g)}c({g_{x}}) \geq \frac{-1}{4}\max_{x\in \Supp(g)} c(g_x) + \left(1-\frac{1}{4}\right)\max_{x\in \Supp(g)} c(g_x) = c'(g,1).$$
This verifies that $\Theta$ is $(F,c',\tilde d_{\max})$--injective, and thus completes the proof of Proposition \ref{prop:metric approx}.

\begin{remark}\label{rem:stronger prop}  Our proof can in fact be subtly modified to give a stronger version of Proposition \ref{prop:metric approx}, 
that is reminiscent of the notion of strong discrete $\calC$--approximations of Holt--Rees \cite{HR}.
Namely, for any $\eta>0$ we can improve the conclusion of part (b) to say that $\Theta$ is $(F,c',\tilde d_{\max})$--injective, where $c'$ is given by
\[c'(g,h)=\begin{cases}
(1-\eta), &\textnormal{ if $h\ne 1$}\\
\max\limits_{x\in \supp(g)}(1-\eta)c({g_{x}}), &\textnormal{ if $h=1,$ $g=(g_{x})_{x\in H}.$}\\
\end{cases}\]
For this improved version, the parameters $E,E_H,E_G,\varepsilon'$ will depend upon $\eta.$ We have elected to not give this improved version in order to simplify the statement of the proposition and its proof.
\end{remark}

\section{Applications of Proposition \ref{prop:metric approx}}\label{sec:apps}

In this section, we  use Proposition \ref{prop:metric approx} to prove Theorem \ref{thmspecial:sofic wreath}.
Part (\ref{I:weak}) of Theorem \ref{thmspecial:sofic wreath}  follows immediately from Proposition \ref{prop:metric approx}, so we focus on proving the remaining three parts.
Each of parts (\ref{I:sofic}),(\ref{I:hyper}),(\ref{I:linear}) are proved below in separate subsections. 
We recall that the aim is to show that, for $H$ a countable, discrete, sofic group, the wreath product $G\wr H$ is respectively sofic, hyperlinear, or linear sofic, whenever $G$ is such a group.

\subsection{Proof of Part \eqref{I:sofic}: Sofic}\label{S:sofic}

We restate and prove our soficity result for wreath products.

\begin{thm}\label{thm:sofic}
	 Let $G,H$ be countable, discrete, sofic groups. Then $G\wr H$ is sofic.
\end{thm}

\begin{proof}
 In order to show that $G\wr H$ is sofic, we show that $G\wr H$ is $\mathcal{C}$-approximable, where $\mathcal{C}$ is the class of symmetric groups with the normalized Hamming distance.
 To do this we compose the map $\Theta$ from Section \ref{sec:construction subsec} with a second map $\Psi$, as described below.
	
Let $F\subseteq G\wr H$ be finite and $\varepsilon>0.$ Let $E_{G},E,E_{H}$ and $\varepsilon'>0$ be as in Proposition \ref{prop:metric approx} for $F,\varepsilon.$ 
\ Define $c$ on $G\setminus \{1\}$ by  $c(g)=\frac{1}{2}$, and so 
\[c'\colon G\wr H\setminus\{1\}\to (0,1/2]\]
as constructed in Proposition \ref{prop:metric approx},
is either ${1}/{2}$ if $h\ne 1$, or ${1}/{4}$ otherwise.

 Since $G,H$ are sofic we can find corresponding sofic approximations.
For $H$ we take $\sigma \colon H \to \Sym(B)$, for a finite set $B$, to be an $(E_H,\varepsilon')$--sofic approximation; 
for $G$ we take $\theta \colon G \to \Sym(A)$, for a finite set $A$, to be an $(E_G,\varepsilon')$--sofic approximation.
Note that, since $\varepsilon' < {1}/{2}$ (see \eqref{eq:epsilon' bounds} following Proposition \ref{prop:metric approx}), the $(E_G,\varepsilon')$--free condition of $\theta$ implies that it is $(E_G,c,d_{\Hamm})$--injective.

With these maps, let $\Theta\colon G\wr H\to\left(\bigoplus_{B}\Sym(A)\right)\wr_{B} \Sym(B)$ be the map constructed in Section \ref{S:construction},  with $K=\Sym(A)$. 
We now explain how we embed $\left(\bigoplus_{B}\Sym(A)\right)\wr_{B} \Sym(B)$ into $\Sym\left(\bigoplus_{B}A\oplus B\right)$.
First, define
\[\Phi\colon \bigoplus_{B}\Sym(A)\to \Sym\left(\bigoplus_{B}A\right)\]
by the diagonal action
\[\Phi((\pi_{\beta})_{\beta\in B}):(a_{\beta})_{\beta\in B}\mapsto (\pi_{\beta}(a_{\beta}))_{\beta\in B}, \ \ \textrm{for $\pi_\beta \in \Sym(A), (a_\beta)_{\beta\in B} \in \bigoplus_B A$.}\]
Then, use $\Phi$ to define the embedding
\[\Psi\colon \left(\bigoplus_{B}\Sym(A)\right)\wr_{B}\Sym(B)\to \Sym\left(\bigoplus_{B}A\oplus B\right)\]
by
$$\Psi(\pi,\tau)\colon (a,b)\mapsto (\Phi(\pi_{\tau(b)})(a),\tau(b))$$
for $\pi \in \bigoplus_{B}\left(\bigoplus_B\Sym(A)\right)$, $\tau\in\Sym(B)$, $a \in \bigoplus_B A$, and $b\in B$. A routine computation reveals that $\Psi$ is a homomorphism.

Let $\ell',\ell_{\max}$ be the  conjugacy-invariant length functions on $\bigoplus_{B}\Sym(A)$ given by
\[\ell'(\pi)=\ell_{\Hamm}(\Phi(\pi)),\] 
\[\ell_{\max}(\pi)=\max_{\beta\in B}\ell_{\Hamm}(\pi_{\beta})\]
for $\pi = (\pi_\beta)_{\beta\in B} \in \bigoplus_B\Sym(A)$.
Then take $\tilde{d},\tilde{d}_{\max}$ to be the bi-invariant metrics as constructed in Proposition \ref{P:conjinvlengthfunc} from the length functions $\ell',\ell_{\max}$ on $\bigoplus_B \Sym(A)$.

	Because $\Psi$ is a homomorphism,  for $\pi_{1},\pi_2\in \bigoplus_{B}\left(\bigoplus_{B}\Sym(A)\right),\tau_{1},\tau_2\in \Sym(B)$, we have:
\begin{equation}\label{eq:sofic dhamm}
d_{\Hamm}(\Psi(\pi_{1},\tau_{1}),\Psi(\pi_{2},\tau_{2}))=\tilde{d}((\pi_{1},\tau_{1}),(\pi_{2},\tau_{2})).
\end{equation}
It thus follows directly from Proposition \ref{prop:metric approx} that $\Psi \circ \Theta$ is $(F,\varepsilon,d_{\Hamm})$--multiplicative. 

 We now show that  $\Psi \circ \Theta$ is $(F,c',d_{\Hamm})$--injective. Let $\pi\in\bigoplus_{B}\left(\bigoplus_{B}\Sym(A)\right)$ and $\tau\in \Sym(B).$ 
Write $\pi=(\pi_{b})_{b\in  B}$ for $\pi_b \in \bigoplus_{B}\Sym(A)$ and, for a fixed $b\in B,$ let $\pi_{b}=(\pi_{b,\beta})_{\beta\in B}.$ 
For each $b\in B$ such that $\tau(b)=b$ we then have
\begin{align*}
\ell_{\Hamm}(\Phi(\pi_b)) 
&= 1 - \frac{1}{\modulus{A}^{\modulus{B}}}\modulus{\left\{ (a_\beta)_{\beta\in B} \mid \pi_{b,\beta} a_\beta = a_\beta \right\}} \\
&= 1 - \frac{1}{\modulus{A}^{\modulus{B}}}\prod_{\beta\in B} \modulus{\left\{ a \in A \mid \pi_{b,\beta} a = a \right\}}\\
&= 1 - \prod_{\beta\in B} \big(1- \ell_{\Hamm}(\pi_{b,\beta})\big)
\end{align*}
which implies
\[\tilde{\ell}\big((\pi,\tau)\big)=\ell_{\Hamm}(\tau)+\frac{1}{|B|}\sum_{\substack{b\in B \\ \tau(b) = b}}\left[1-\prod_{\beta\in B}\big(1-\ell_{\Hamm}(\pi_{b,\beta})\big)\right].\]
Since $0\leq \ell_{\Hamm}(\pi_{b,\beta})\leq 1$ we have for each $b\in B:$
\[\prod_{\beta\in B}\big(1-\ell_{\Hamm}(\pi_{b,\beta})\big)\leq 1-\max_{\beta\in B}\ell_{\Hamm}(\pi_{b,\beta})\]
and inserting this into the above expression for $\tilde{d}$ shows that $$\tilde{\ell}((\pi,\tau))\geq \tilde{\ell}_{\max}((\pi,\tau)).$$
Combining equation \eqref{eq:sofic dhamm} with the preceding inequality, we get for each $g\in F\setminus\{1\}$
\[\ell_{\Hamm}\big(\Psi(\Theta(g))\big)=\tilde{\ell}\big(\Theta(g)\big)\geq \tilde{\ell}_{\max}\big(\Theta(g)\big)\geq c'(g)\]
since Proposition \ref{prop:metric approx} implies $\Theta$ is $(F,c',\tilde d_{\max})$--injective.
This shows that $\Psi \circ \Theta$ is $(F,c',d_{\Hamm})$--injective. Hence we have shown that $G\wr H$ is $\mathcal{C}$--approximable, where $\mathcal{C}$ is the class of symmetric groups equipped with the Hamming distance and this means that $G\wr H$ is sofic.
\end{proof}

We remark that one can use the improved version of Proposition \ref{prop:metric approx}, 
as per Remark \ref{rem:stronger prop}, to show that $\Psi \circ \Theta$ as considered in the above proof is an $(F,\varepsilon)$--sofic approximation provided $\theta\colon G\to \Sym(A)$ and $\sigma\colon H\to \Sym(B)$ are sufficiently good sofic approximations. In this way one can in fact directly show that $G\wr H$ has arbitrarily good sofic approximations.

\subsection{Proof of Part \eqref{I:hyper}: Hyperlinear}

In this section, we deduce hyperlinearity of $G\wr H,$ assuming that $G$ is hyperlinear and $H$ is sofic. 
Hyperlinear groups are defined by admitting a metric approximation to unitary groups, $\mathcal{U}(n)$, paired with the normalized Hilbert-Schmidt metric.

Let $\tr\colon M_{n}(\C)\to \C$ be the normalized trace:
\[\tr(A)=\frac{1}{n}\sum_{j=1}^{n}A_{jj}\]
where $A = (A_{ij}) \in M_n(\C)$.

\begin{defn}
	The \emph{normalized Hilbert-Schmidt norm} on $M_{n}(\C)$ is defined by $$\|A\|_{2}=\tr(A^{*}A)^{1/2}, \ \ \textrm{for $A\in M_n(\C)$.}$$ 
	The \emph{normalized Hilbert-Schmidt metric} on $\calU(n)$ is therefore given by 
	$$\dhs(U,V) = \norm{U-V},\ \ \textrm{for $U,V \in \calU(n)$.}$$
	The corresponding length function is denoted $\lhs$.
\end{defn}

\begin{defn}\label{def:hyper}
	We say a group is \emph{hyperlinear} if it is $\mathcal{C}$--approximable, where $\mathcal{C}$ is the class of unitary groups, paired with the normalized Hilbert-Schmidt metrics.
\end{defn}

We will need that our approximations $\theta\colon G\to \mathcal{U}(n)$ not only map $\theta(g)$ far away from $\id$ for $g \ne 1,$ but that in fact $\theta(g)$ is far away from  the unit circle $S^{1}=\{\lambda \id: \modulus{\lambda}=1\}$ in $\calU(n)$. 
To put this in a framework where we can take advantage of Proposition \ref{prop:metric approx}, 
we use the following set-up.

Define $\dbarhs$, a bi-invariant metric on $\calU(n)/S^1$, by
	$$\dbarhs(US^1,VS^1) = \inf_{\lambda\in S^1} \dhs(\lambda U, V),\ \ \textrm{for $U,V \in \calU(n)$.}$$
Let $\lbarhs$ denote the corresponding length function.
We will abuse this notation and write $\dbarhs(U,V)$.
Note that we can directly use the normalized trace to calculate $\lbarhs(U)$ as follows: 
\[\lbarhs(U)^{2}
=\inf_{\lambda\in S^{1}}\norm{U-\lambda\id}^{2}
=\inf_{\lambda\in S^{1}}2-2\rea(\overline{\lambda}\tr(U))
=2-2\modulus{\tr(U)}.\]
In light of this, we get the following reformulation of a result of R\u{a}dulescu in \cite{Radul}  which gives an equivalent definition of hyperlinearity.

\begin{prop}\label{P:projunitary}
Let $G$ be a group and $c\colon G\setminus \{1\}\to (0,\sqrt{2})$ any function.

Then $G$ is hyperlinear if and only if for every $\varepsilon>0$ and any finite $F\subseteq G$ there is {a positive integer $n$ and} a function $\theta\colon G\to \mathcal{U}(n)$ which is $(F,\varepsilon,\dhs)$--multiplicative and so that $q\circ \theta$ is $(F,c,\dbarhs)$--injective, where $q\colon \mathcal{U}(n)\to \mathcal{U}(n)/S^{1}$ is the quotient map. 
\end{prop}

\begin{thm}\label{thm:hyperlinear wreath sofic}
	 Let $H$ be a countable, discrete, sofic group and $G$ a countable, discrete, hyperlinear group. Then $G\wr H$ is hyperlinear.
\end{thm}

\begin{proof}
We proceed in an analogous manner as for Theorem \ref{thm:sofic}, when we dealt with soficity. 
In particular, we show that $G\wr H$ is $\calC$--approximable, where $\calC$ is as in Definition \ref{def:hyper}. 
The necessary maps to demonstrate this will be constructed as a composition, starting with $\Theta$ from Proposition \ref{prop:metric approx} 
 followed by an appropriate embedding 
  into a unitary group.

\medskip

{\it Step 1:} Setting the scene.

Let $F\subseteq G\wr H$ be finite and $\varepsilon>0.$ Let $E_{G},E,E_{H}$ and $\varepsilon'>0$ be as Proposition \ref{prop:metric approx} for $F,\varepsilon.$ Let 
$c\colon G\setminus \{1\}\to (0,1/2]$
be given by $c(g)=\frac{1}{2}$ for $g\in G\setminus \{1\}$ and let 
$c'\colon G\wr H\setminus\{1\}\to (0,1/2]$
be the map constructed in Proposition \ref{prop:metric approx}. 
Since $H$ is sofic we can find an $(E_{H},\varepsilon')$--sofic approximation $\sigma\colon H\to \Sym(B)$ for some finite set $B.$ 
Since $G$ is hyperlinear we apply Proposition \ref{P:projunitary} to find an $(E_{G},\varepsilon,\dhs)$--multiplicative map $\theta\colon G\to \mathcal{U}(\mathcal{H})$ for some finite-dimensional Hilbert space $\mathcal{H}$ so that $q\circ \theta$ is $(E_{G},c,\dbarhs)$--injecitve.

Let $\Theta\colon G\wr H\to\left(\bigoplus_{B}\mathcal{U}(\mathcal{H})\right)\wr_{B}\rtimes \Sym(B)$ be the map constructed from $\theta, \sigma$ and $E$ in Section \ref{S:construction}. Similarly construct $\bar \Theta : G\wr H \to \left(\bigoplus_{B}\mathcal{U}(\mathcal{H})/S^{1}\right)\wr_{B}\Sym(B)$ from $q\circ \theta,\sigma$ and $E$.

Define  
		$$\Phi\colon \bigoplus_{B}\mathcal{U}(\mathcal{H})\to \mathcal{U}(\mathcal{H}^{\otimes B})$$
  by
		$$\Phi \colon (V_{\beta})_{\beta\in B}  \mapsto \bigotimes_{\beta\in B}V_{\beta}, \ \ \textrm{for $(V_\beta)_{\beta \in B} \in \bigoplus_B  \calU(\calH)$.}$$
We now define 
		$$\Psi\colon \left(\bigoplus_{B} \mathcal{U}(\mathcal{H}) \right) \wr_{B} \Sym(B) \to \mathcal{U}\left( \bigoplus_B \left(\mathcal{H}^{\otimes B}\right)\right)$$
 by
		$$\Psi((U_{b})_{b\in B},\tau) \colon (\xi_{b})_{b\in B} \mapsto \left(\Phi (U_{b})\left(\xi_{\tau^{-1}(b)}\right)\right)_{b\in B}$$
for $(\xi_b)_{b\in B} \in \bigoplus_B\left(\mathcal{H}^{\otimes B}\right)$,  $(U_b)_{b \in B} \in \bigoplus_B\bigoplus_B \calU(\calH)$, and $\tau \in \Sym(B)$.
The collection of maps we have is summarized in Figure \ref{fig:map of maps hyp}.

\begin{figure}[h!]
	\begin{tikzcd}[column sep = large, row sep=large]
		\displaystyle G\wr H \arrow{r}{\Theta}\arrow{rd}[swap]{\bar\Theta}
		&\displaystyle \left(\bigoplus_{B} \mathcal{U}(\mathcal{H}) \right) \wr_{B} \Sym(B)  \arrow{d}\arrow{r}{\Psi}
		&\displaystyle \mathcal{U}\left( \bigoplus_B \left(\mathcal{H}^{\otimes B}\right)\right)\\
		&\displaystyle \left( \bigoplus_B \calU(\calH) / S^1 \right) \wr_B \Sym(B)
	\end{tikzcd}
	\caption{A plan of the maps involved.}\label{fig:map of maps hyp}
\end{figure}

Let $\tilde{d},\tilde{d}_{\max}$ be the bi-invariant metrics on $\left(\bigoplus_{B}\mathcal{U}(\mathcal{H})\right)\wr_{B}\Sym(B)$ and $\left(\bigoplus_{B}\mathcal{U}(\mathcal{H})/S^{1}\right)\wr_{B}\Sym(B)$,
respectively,
induced by Proposition \ref{P:conjinvlengthfunc} from the length functions $\ell',\ell_{\max}$
on $\bigoplus_{B}\mathcal{U}(\mathcal{H})$ and $\bigoplus_{B}\mathcal{U}(\mathcal{H})/S^{1}$, respectively, which are given by
\[\ell'(V)=\frac{1}{2}\lhs(\Phi(V)),\]
\[\ell_{\max}(\bar V)=\max_{\beta\in B}\frac{\lbarhs(q (V_{\beta}))}{\sqrt{2}},\]
for $V=(V_\beta)_{\beta\in B} \in \bigoplus_{B}\mathcal{U}(\mathcal{H})$,
and $\bar{V}=(q(V_\beta))_{\beta\in B} \in \bigoplus_{B}\mathcal{U}(\mathcal{H})/S^{1}$. Note that $\dhs$ is bounded by $2,$ whereas $\dbarhs$ is bounded by $\sqrt{2}.$

\medskip

{\it Step 2:} A formula for $\dhs(\Psi(U,\tau),\id)$.

 We aim to bound the $\dhs$--distance from a point in the image of $\Psi$ to the identity in terms of the $\tilde{d}$--distance for its pre-image. 
To this end, we first observe that the matrix representation for $\Psi(U,\tau)$ will be a block permutation matrix, with blocks corresponding to elements of $B$.
The matrix will have a non-zero block in the $(b,b)$--position precisely when $\tau(b)=b$. Thus we get
		$$\tr\left(\Psi(U,\tau)\right) =\frac{1}{\modulus{B}} \sum_{\substack{b\in B\\ \tau(b)=b}} \tr \left(\Phi (U_b)\right),\ \ \textrm{where $U=(U_b)_{b\in B}$.}$$
This implies that 
		$$\norm{\Psi(U,\tau)-\id}^{2} = 2-2\rea(\tr(\Psi(U,\tau))) =
2-\frac{2}{\modulus{B}}\sum_{\substack{b\in B\\\tau(b)=b}}\rea(\tr(U_{b})).$$
By the definition of the Hamming metric we can rewrite the right-hand side as
		$$2\ell_{\Hamm}(\tau)+\frac{2}{\modulus{B}}\sum_{\substack{b\in B\\\tau(b)=b}}1-\rea(\tr(\Phi(U_{b}))).$$
Hence 
		\begin{equation} \label{eq:norm and trace sum}
		\norm{\Psi(U,\tau)-\id}^{2}  = 2\ell_{\Hamm}(\tau)+\frac{2}{\modulus{B}}\sum_{\substack{b\in B\\\tau(b)=b}}\norm{\Phi(U_{b})-\id}^{2}.
		\end{equation}
		
\medskip

{\it Step 3:} Almost multiplicativity.

Since $\norm{\Phi(U_b)-\id} \leq \sqrt{2}$, we can get an upper bound of
		$$\norm{\Psi(U,\tau)-\id}^{2}\leq 2\ell_{\Hamm}(\tau)+\frac{4}{|B|}\sum_{\substack{b\in B\\\tau(b)=b}}\norm{\Phi(U_{b})-\id}\leq 8\tilde{d}((U,\tau),1).$$
In summary, since $\Psi$ is a homomorphism, for $(U_{1},\tau_{1}), (U_{2},\tau_{2})\in \left(\bigoplus_{B}\mathcal{U}(\mathcal{H})\right)\wr_{B}\Sym(B)$ we have shown
		\[\dhs(\Psi(U_{1},\tau_{1}),\Psi(U_{2},\tau_{2}))\leq 2\sqrt{2}\tilde{d}((U_{1},\tau_{1}),(U_{2},\tau_{2}))^{1/2}.\]
From Proposition \ref{prop:metric approx} we know that $\Theta$ is $(F,\varepsilon,\tilde d)$--multiplicative. With this, the above inequality then implies that $\Psi\circ \Theta$ is $(F,2\sqrt{2\varepsilon},\dhs)$--multiplicative.

\medskip 

{\it Step 4:} Almost injectivity.

Let $V=(V_{\beta})_{\beta\in B}\in \bigoplus_{B}\mathcal{U}(\mathcal{H})$. For each $\beta$, we have that 
		\[\modulus{\tr(V_\beta)}=1-\left(\frac{\lbarhs (q(V_\beta))}{\sqrt{2}}\right)^{2}.\]
Thus
		\[\modulus{\tr(\Phi(V))}=\prod_{\beta\in B}\modulus{\tr(V_{\beta})}\leq 1-\max_{\beta\in B}\left(\frac{\lbarhs(q(V_{\beta}))}{\sqrt{2}}\right)^{2} = 1 - \ell_{\max}(\bar V)^2\]
where $\bar V = (q(V_\beta))_{\beta\in B}$. 
 Since $2-2\modulus{\tr(\Phi(V))} \leq \norm{\Phi(V)-\id}^2$, we get $\ell_{\max}(\bar V)^2 \leq \frac{1}{2}\norm{\Phi(V)-\id}^2$.
Inserting this into equation \eqref{eq:norm and trace sum} and arguing as in Section \ref{S:sofic} we see that, if  $\bar U_b = (q(U_{b,\beta}))_{\beta\in B}$, then 
\begin{align*}
\norm{\Psi(U,\tau)-\id}^{2}
&\geq 2\ell_{\Hamm}(\tau)+\frac{2}{|B|}\sum_{\substack{b\in B\\\tau(b)=b}}\ell_{\max}(\bar U_b)^{2}\\
&\geq 2\ell_{\Hamm}(\tau)^{2}+2\left(\frac{1}{|B|}\sum_{\substack{b\in B\\\tau(b)=b}}\ell_{\max}(\bar U_b)\right)^{2}\\
&\geq \left(\ell_{\Hamm}(\tau)+\frac{1}{|B|}\sum_{\substack{b\in B\\\tau(b)=b}}\ell_{\max}(\bar U_b)\right)^{2}\\
&=\tilde{d}_{\max}\big((\bar U,\tau),1\big)^{2}
\end{align*}
 where $\bar U = (\bar U_b)_{b\in B}$.
As  $q\circ\theta$ is $(E_{G},c,\dbarhs)$--injective, it follows by Proposition \ref{prop:metric approx} that $\bar \Theta$ is $(F,c',\tilde d_{\max})$--injective.
Thus, for $(U,\tau)$ in the image of $\Theta$, it follows that $(\bar U, \tau)$ is in the image of $\bar \Theta$, and
\[\lhs(\psi(U,\tau))^2 = \norm{\Psi(U,\tau)-\id}^{2}\geq (c'(x))^2.\]
Thus  $\Psi \circ \Theta$ is $(F,c',\dhs)$--injective and $(F,2\sqrt{\varepsilon},\dhs)$--multiplicative. As $\varepsilon>0$ is arbitrary the proof is complete.
\end{proof}

As with soficity, one can use the improved version of Proposition \ref{prop:metric approx} from Remark \ref{rem:stronger prop} to strengthen the bounds in the above results.
In particular this will show that 
\[\min_{x\in F\setminus\{1\}}\lbarhs(\Psi(\Theta(x)))\geq 1-\varepsilon,\]
provided $\sigma\colon H\to \Sym(B)$ is a sufficiently good sofic approximation and $\theta$ satisfies  \[\min_{g\in E}\lbarhs(\theta(g))>1-\kappa,\]
for a sufficiently large $E$ and a sufficiently small $\kappa.$ 
In this manner, we can directly verify the conclusion of Proposition \ref{P:projunitary} for $G\wr H$ if $H$ is sofic and $G$ is hyperlinear.

\subsection{Proof of Part \eqref{I:linear}: Linear Sofic}

We recall the following definition due to Arzhantseva and Paunescu \cite{LinearSofic}.

\begin{defn} Let $\FF$ be a field. Define a bi-invariant metric $\drk$, with corresponding length function $\lrk$, on $\GL_{n}(\FF)$ by
$$\drk(A,B)=\frac{1}{n}\rank(A-B).$$
We say that a group is \emph{linear sofic over $\FF$} if it is $\mathcal{C}$-approximable, where $\mathcal{C}$ consists of all general linear groups $\GL_n(\FF)$, each paired with the metric $\drk$.
\end{defn}
In this section we use Proposition \ref{prop:metric approx} to show that $G\wr H$ is linear sofic if $G$ is linear sofic and $H$ is sofic. Proving that the map we constructed is sufficiently injective turns out to be trickier than in any of the other cases.
As in the case of hyperlinear groups,  we will need that  our linear sofic approximation $\theta\colon G\to \GL_{n}(\FF)$ does not just satisfy that  $\frac{1}{n}\rank(\theta(g)-\id)$ is bounded away from $0$ for $g\ne 1,$ but in fact we need
\[\min_{\lambda\in \KK^{\times}}\frac{1}{n}\rank_{\KK}(\theta(g)-\lambda \id)> 0,\]
where $\rank_{\KK}$ indicates that we are computing dimension over the algebraic closure $\KK$ of $\FF$.  Thus we use the following definition.

\begin{defn} Let $\FF$ be a field, for $A,B\in \GL_{n}(\FF),$ we let
\[\drkbar(A,B)=\min_{\lambda\in \KK^{\times}}\frac{1}{n}\rank(A-\lambda B)\]
and $\lrkbar$ denote the corresponding length function.
\end{defn}

Note that, since $\frac{1}{n}\rank_{\FF}(A-B)=\frac{1}{n}\rank_{\KK}(A-B)$ for $A,B\in \GL_{n}(\FF),$ we have $\drk(A,B)\geq \drkbar(A,B).$ 

We will then use the following fact, which is a consequence of an equivalent characterization of linear soficity given by Arzhantseva--Paunescu \cite[Theorem 5.10]{LinearSofic}.

\begin{prop}\label{P:projlinearsofic} 
	Let $G$ be a linear sofic group over the field $\FF$ and let $\KK$ denote the algebriac closure of $\FF.$
	
	Then, for any $\delta\in (0,\frac{1}{8})$ and any finite $F\subseteq G,$ there is  a positive integer $n$ and a function $\theta\colon G\to \GL_{n}(\FF)$ 
	which is $(F,\delta,\drk)$--multiplicative, and so that $q\circ \theta$ is $(F,c,\drkbar)$--injective, 
	where $c(g)=\frac{1}{8} -\delta$ for all $g\in G$, and $q \colon \GL_n(\FF) \to \PGL_n(\KK)$ is the canoncial map given by composing the natural inclusion $\GL_{n}(\FF)\to \GL_{n}(\KK)$ with the quotient map $\GL_{n}(\KK)\to \PGL_{n}(\KK).$ 
\end{prop}

\begin{proof} By \cite[Theorem 5.10]{LinearSofic}, it follows that there exists a function
\[\theta_0 \colon G\to \GL_{m}(\FF)\]  for some $m\in \N$, 
which is $(F,\delta,\drk)$--multiplicative and so that $\drk(\theta_0(g)-\id)\geq \frac{1}{4}-2\delta$ for all $g\in F\setminus\{1\}.$ Now consider 
\[\theta\colon G\to \GL_{2m}(\FF)\]
given in matrix block form by
\[
{\theta}(g)=\begin{bmatrix}
\theta_0(g) & 0\\
0 &   \id
\end{bmatrix}.\]
Fix $\lambda\in \KK^{\times}$ and $g\in F\setminus\{1\}.$
If $\lambda\ne 1,$ then we see that $\drk({\theta}(g),\lambda\id)\geq \frac{1}{2}$.
On the other hand, if $\lambda=1$  then 
\[\frac{1}{2m}\rank_{\KK}({\theta}(g)-\lambda \id)=\frac{1}{2}\cdot\left[\frac{1}{m}\rank_{\FF}(\theta_0(g)-\id)\right]\geq \frac{1}{8}-\delta.\]
Thus $\theta$ is the required function.
\end{proof}

In order to use  Proposition \ref{prop:metric approx} to prove that $G\wr H$ is linear sofic, we will need to use tensor products of matrices. The main fact we will need is that if $A\in \GL_{n}(\FF),B\in \GL_{k}(\FF)$ and $\lrkbar(A),\lrkbar(B)$ are both bounded away from zero, then $\lrkbar(A\otimes B)$ is also bounded away from zero. We formulate this precisely in the Proposition \ref{prop:tesnorlowerbound} below, whose proof uses similar ideas to \cite[Lemma 5.4, Prop 5.8]{LinearSofic}.  

Let $J_\alpha(A)$ denote the number of Jordan blocks in the Jordan normal form of $A$ associated to the eigenvalue $\alpha$. If $\alpha$ is not an eigenvalue then we set $J_\alpha(A)=0$.
Given a number $\alpha$ and a positive integer $n$ we let $J(\alpha,n)$ denote the standard $n \times n$ Jordan block with eigenvalue $\alpha$. 
In characteristic zero,  the following is a classic result explaining how Jordan blocks behave under tensor products, known as the Clebsch-Gordan formula (and in fact one can even say  what the precise Jordan block decomposition of $J(\alpha,n)\otimes J(\beta,k)$ is, though we will not need this). See, for example, \cite[Theorem 2]{MatrsinVlass}. For positive characteristic, this result is a consequence of \cite[Theorem 2.2.2]{IliIwa}.

\begin{thm}\label{thm:tensor jordans}
	Let $\KK,$ be an algebraically closed field, $\alpha, \beta$ be nonzero elements of $\KK$, and $n,k$ be positive integers.
	Then 
	$$J_{\alpha\beta}\big(J(\alpha,n) \otimes J(\beta,k)\big) = \min\{n,k\}.$$
\end{thm}

We will use this to prove the following.

\begin{lem}\label{lem:counting jordan blocks in tensor products}
	Let $\KK$ be an algebraically closed field and take $A \in \GL_n(\KK)$ and  $B \in \GL_k(\KK)$.
		
	Then, for each $\lambda \in \KK$, 
		$$J_{\lambda}(A\otimes B) \leq \min\left\{ k \max\limits_{\alpha \in \KK} J_\alpha(A) , n \max\limits_{\beta \in \KK} J_\beta(B) \right\}.$$
\end{lem}

\begin{proof}
	Let us first prove that $J_{\lambda}(A\otimes B)\leq k\max_{\alpha\in \KK}J_{\alpha}(A),$ as the other inequality will follow by symmetry. First, assuming that $A$ and $B$ have unique eigenvalues $\alpha$ and $\beta$ respectively,
	the result of the lemma becomes
		\begin{equation}\label{eq:jordan block claim}
			J_{\alpha\beta}(A\otimes B) \leq kJ_{\alpha}(A).
		\end{equation}
Both sides of the above inequality are additive under taking direct sums of matrices with the given eigenvalues. So we may assume that $A$ and $B$ are one Jordan block, in which case  \eqref{eq:jordan block claim} follows from Theorem \ref{thm:tensor jordans}.
	
	Now suppose the eigenvalues of $A$ and $B$ are not necessarily unique.
	Since $\KK$ is algebraically closed, up to conjugacy we may write $A$ and $B$ as direct sums
			$$A = \bigoplus_{\alpha \in \FF} A_\alpha,\ \ B = \bigoplus_{\beta\in \KK} B_\beta$$
	where $A_\alpha$ is the direct sum of all Jordan blocks of $A$ associated to eigenvalue $\alpha$, and similarly for $B_\beta$.
	Suppose $A_\alpha$ is $n_\alpha \times n_\alpha$ and $B_\beta$ is $k_\beta \times k_\beta$.
	Then
		$$A\otimes B = \bigoplus_{\alpha,\beta\in \KK} A_\alpha \otimes B_\beta,$$
	which leads to the following, using \eqref{eq:jordan block claim}:
		$$J_\lambda(A\otimes B) 
		= \sum\limits_{\alpha\beta=\lambda} J_\lambda(A_\alpha \otimes B_\beta) 
		\leq \sum\limits_{\alpha\beta=\lambda} k_\beta  J_\alpha(A)
		\leq \sum\limits_{\beta\in \KK} k_\beta \max\limits_{\alpha \in \KK} J_\alpha(A) 
		= k \max\limits_{\alpha \in \KK} J_\alpha(A).$$
This completes the proof.
\end{proof}

To see how the normalized rank metric $\drkbar$ behaves under tensor products, we remark that  $\rank(A-\alpha\id) = n - J_\alpha(A)$ for every $\alpha\in \KK$, implying 
$\drkbar (A,\id) = \inf_{\lambda\in \KK} \left( 1- \frac{1}{n}J_\lambda(A) \right)$.
The following is thus an immediate consequence of this fact, and of Lemma \ref{lem:counting jordan blocks in tensor products}.

\begin{prop}\label{prop:tesnorlowerbound} Let $\FF$ be a field. Let $n,k\in \N$ and $A\in \GL_{n}(\FF),B\in \GL_{k}(\FF).$ Then
\[\drkbar(A\otimes B,\id)\geq \max\left\{ \drkbar(A,\id),\drkbar(B,\id) \right\}.\]
\end{prop}

\begin{thm}
	Let $G$ be a linear sofic group over the field $\FF$ and $H$ be a sofic group. Then $G\wr H$ is linear sofic over $\FF.$
\end{thm}

\begin{proof}
	The structure of the proof is analogous to that of Theorem \ref{thm:hyperlinear wreath sofic}. We compose the map $\Theta : G\wr H \to \left(\bigoplus_{B}\GL_{n}(\FF)\right)\wr_{B}\Sym(B)$ from Proposition \ref{prop:metric approx} with a map $\Psi$ giving us a map from $G \wr H$ to a linear group. We verify that $\Psi \circ \Theta$ satisfies the required almost multiplicativity and almost injectivity conditions.
	
\medskip

{\it Step 1:} Setting the scene.

Recall that $q\colon \GL_{n}(\FF)\to \PGL_{n}(\KK)$ denotes the composition of the canonical inclusion $\GL_{n}(\FF)\to \GL_{n}(\KK)$, where $\KK$ is the algebraic closure of $\FF$, with the quotient map $\GL_{n}(\KK)\to \PGL_{n}(\KK).$ 

Take a finite subset $F$ of $G\wr H$ and $\varepsilon>0.$
Define $c\colon G\setminus\{1\}\to (0,\infty)$ to take the value $\frac{1}{16}$ for all $g \neq 1$.
Let $E_{G}\subseteq G,E,E_{H}\subseteq H$, $c' : G\setminus \{1\} \to (0,\infty)$,  and $\varepsilon'>0$ all be as determined by $F,\varepsilon,$ and $c$ in Proposition \ref{prop:metric approx}.
Note that  from \eqref{eq:epsilon' bounds} in the proof of Proposition \ref{prop:metric approx}, we know that $\varepsilon' < \frac{1}{16^2}<\frac{1}{16}$.
Thus, taking $\delta=\varepsilon'$ in Proposition \ref{P:projlinearsofic} gives us a map $\theta \colon G\to \GL_{n}(\FF)$ that is $(E_{G},\varepsilon',\drk)$--multiplicative and is such that $q\circ \theta$ is  $(E_{G},c,\drkbar)$--injective.

Let $\sigma\colon H\to \Sym(B),$ for some finite set $B,$ be an $(E_{H},\varepsilon')$--sofic approximation and take 
$$\Theta\colon G\wr H\to \left(\bigoplus_{B}\GL_{n}(\FF)\right)\wr_{B}\Sym(B)$$
to be the map constructed from $\theta,\sigma$, and $E$ in Section \ref{S:construction}. 
Meanwhile, let 
$$\bar{\Theta}  \colon G\wr H \to \left(\bigoplus_{B}\PGL_{n}\left(\KK\right)\right)\wr_{B}\Sym(B)$$
be the map constructed using $q\circ\theta$ in place of $\theta$.

We now describe how to embed the image of $\Theta$ into a linear group. First define
\[\Phi\colon \bigoplus_{B}\GL_{n}(\FF)\to \GL\left((\FF^{n})^{\otimes B}\right)\]
by
\[\Phi \colon (X_{\beta})_{\beta\in B} \mapsto \bigotimes_{\beta\in B}X_{\beta},\ \  \textrm{for $(X_\beta)_{\beta\in B} \in \bigoplus_B \GL_n(\FF)$.}\]
Using $\Phi$, we define 
	$$\Psi\colon \left(\bigoplus_{B}\GL_{n}(\FF)\right)\wr_{B}\Sym(B)\to \GL\left({\bigoplus_B}\left(\left(\FF^{n}\right)^{\otimes B}\right)\right)$$
by
	$$\Psi((A_{b})_{b\in B},\tau) \colon (\xi_{b})_{b\in B} \mapsto \left(\Phi(A_{b})\left(\xi_{\tau^{-1}(b)}\right)\right)_{b\in B}$$
for $(\xi_b)_{b \in B} \in {\bigoplus_B}\left(\left(\FF^{n}\right)^{\otimes B}\right)$, $(A_b)_{b\in B} \in \bigoplus_B\bigoplus_B \GL_n(\FF)$, and $\tau \in \Sym(B)$.

The collection of maps we have is summarized in Figure \ref{fig:map of maps}.

\begin{figure}[h!]
	\begin{tikzcd}[column sep = large, row sep=large]
			\displaystyle G\wr H \arrow{r}{\Theta}\arrow{rd}[swap]{\bar\Theta}
			&\displaystyle \left(\bigoplus_B\GL_n(\FF)\right)\wr_B \Sym(B) \arrow{d}\arrow{r}{\Psi}
			&\displaystyle \GL\left({\bigoplus_B}\left(\left(\FF^{n}\right)^{\otimes B}\right)\right)\\
			&\displaystyle \left( \bigoplus_B \PGL_n\left(\KK\right)\right) \wr_B \Sym(B)
	\end{tikzcd}
	\caption{A plan of the maps involved.}\label{fig:map of maps}
\end{figure}

Let $\widetilde{d},\tilde d_{\max}$ be the bi-invariant metrics on the wreath products
$\left(\bigoplus_{B}\GL_n(\FF)\right)\wr_{B}\Sym(B),$ and
$\left(\bigoplus_{B}\PGL_n(\KK)\right)\wr_{B}\Sym(B)$, 
respectively, obtained by applying Proposition \ref{P:conjinvlengthfunc} to the length functions
$\ell',\ell_{\max}$  on $\bigoplus_{B}\GL_n(\FF),$ and $\bigoplus_{B}\PGL_n(\KK)$, respectively,
given by
\[\ell'(X)=\lrk(\Phi(X)),\]
\[\ell_{\max}(\bar X)=\max_{\beta\in B}\lrkbar(\Phi((X_{\beta})_{\beta\in B})),\]
where $X = (X_\beta)_{\beta \in B}\in \bigoplus_B\GL_n(\FF)$, and $\bar X = (q(X_\beta))_{\beta\in B} \in \bigoplus_B \PGL_n(\KK)$.

\medskip

{\it Step 2:} 
A formula for $\lrk(\Psi(A,\tau))$.

We wish to show that $\Psi\circ\Theta$ is almost multiplicative and almost injective. To do this we need a good handle on $\lrk(\Psi(A,\tau))$ when $(A,\tau)$ is in the image of $\Theta$.

Write $A=(A_{b})_{b\in B}$ with $A_{b}\in \bigoplus_{B}\GL_{n }(\FF).$
The kernel of $\Psi(A,\tau)-\id$ is given by
\[
\left\{(\xi_b)_{b\in B}\in \bigoplus_{\substack{b\in B\\\tau(b)\neq b}} \left(\FF^n\right)^{\otimes B} 
\colon 
\Phi(A_{\tau(b)})(\xi_{b})=\xi_{\tau(b)}\right\}
\oplus
\left(\bigoplus_{\substack{b\in B\\\tau(b)=b}}\ker(\Phi(A_{b})-\id)\right)
.\]
Focusing on the left term in the above direct sum, if we pick a cycle $(b_1\ b_2 \cdots b_k)$ of $\tau$, with $k\geq 2$, then $\xi_{b_1}$ determines $\xi_{b_i}$ for $i=2,\ldots,k$.
Thus each cycle of length greater than 1 contributes exactly $n^{\modulus{B}}$ to the dimension of the kernel.
Let  $\cyc_0(\tau)$ be the number of cycles of length at least two in the cycle decomposition of $\tau$. 
From the above discussion we see that the dimension of $\ker(\Psi(A,\tau)-\id)$ is
\[
n^{\modulus{B}}  \cyc_0(\tau)
+
\sum_{\substack{b\in B\\\tau(b)= b}}\dim(\ker(\Phi(A_{b})-\id))
.\]
It follows that
\begin{eqnarray*}
	\lrk(\Psi(A,\tau))
	&=&
	1-\frac{\dim(\ker(\Psi(A,\tau)-\id)}{n^{\modulus{B}}\modulus{B}}\\
	&=&
	1 - \frac{\cyc_0(\tau)}{\modulus{B}} 
	- \sum_{\substack{b\in B\\\tau(b)= b}}\frac{1 - \lrk(\Phi(A_b))}{\modulus{B}}.
\end{eqnarray*}
Since
\[\ell_{\Hamm}(\tau) = 1- \frac{\modulus{\{b \in B:\tau(b)=b\}}}{\modulus{B}}\]
we get
\begin{equation}\label{E:rankcomputation}
\lrk(\Psi(A,\tau))=
	\ell_{\Hamm}(\tau) - \frac{\cyc_0(\tau)}{\modulus{B}} 
	+ \frac{1}{\modulus{B}} \sum_{\substack{b\in B\\\tau(b)= b}} \lrk(\Phi(A_b)).
\end{equation}

\medskip
{\it Step 3:} Almost multiplicativity.

Equation \eqref{E:rankcomputation} implies that 
\[\lrk(\Psi(A,\tau),)\leq \tilde{\ell}((A,\tau)).\]
Bi-invariance implies that for $(A_{1},\tau_{1}),(A_{2},\tau_{2})\in \left(\bigoplus_{B}\GL_{n}(\FF)\right)\wr_{B}\Sym(B)$ we have:
\[\drk(\Psi(A_{1},\tau_{1}),(A_{2},\tau_{2}))\leq \tilde{d}((A_{1},\tau_{1}),(A_{2},\tau_{2})).\]
Thus $(F,\varepsilon,\drk)$--multiplicativity of $\Psi\circ \Theta$ follows from the $(F,\varepsilon,\tilde d)$--multiplicativity  of $\Theta$. 

\medskip
{\it Step 4:} Almost injectivity.

While for almost multiplicativity we used the almost multiplicativity of $\Theta$, for almost injectivity we will use the almost injectivity of $\bar{\Theta}$.

 Elementary calculations yield
	$$\ell_{\Hamm}(\tau)
	 = 
	\frac{\modulus{B}- \modulus{b\in B \colon \tau(b)=b}}{\modulus{B}}
	\geq 
	\frac{2 \cyc_0(\tau)}{\modulus{B}}.
	 $$
Using this in \eqref{E:rankcomputation}, we get that
\[\lrk(\Psi(A,\tau))\geq \frac{1}{2}\ell_{\Hamm}(\tau)+\frac{1}{\modulus{B}}\sum_{\substack{b\in B\\\tau(b)=b}}\lrk(\Phi(A_{b})).\]
By repeated applications of Proposition \ref{prop:tesnorlowerbound} we have, for each $b\in B,$
\[\lrk(\Phi(A_{b}))\geq \max_{\beta\in B}\lrkbar(A_{b,\beta}).\]
This implies that
\[\lrk(\Psi(A,\tau))\geq \frac{1}{2} \tilde \ell_{\max}((\bar A,\tau)).\]
where $\bar A = ((q(A_{b,\beta})_{\beta\in B})_{b\in B}$.
If $(A,\tau)$ lies in the image of $\Theta$ then $(\bar A,\tau)$ lies in the image of $\bar \Theta$. Then, $(F,c',\tilde d_{\max})$--injectivity of $\bar \theta$, coupled with the above inequality, gives us $(F,\frac{c'}{2},\drk)$--injectivity of $\Psi\circ\Theta$.
\end{proof}

\bibliographystyle{alpha}
\bibliography{bibliography_sofic}

\begin{thebibliography}{DKP14}

\bibitem[APa17]{LinearSofic}
Goulnara Arzhantseva and Liviu P\u~aunescu.
\newblock Linear sofic groups and algebras.
\newblock {\em Trans. Amer. Math. Soc.}, 369(4):2285--2310, 2017.

\bibitem[Bow10]{Bow}
Lewis Bowen.
\newblock Measure conjugacy invariants for actions of countable sofic groups.
\newblock {\em J. Amer. Math. Soc.}, 23(1):217--245, 2010.

\bibitem[CHR14]{CHR}
Laura Ciobanu, Derek~F. Holt, and Sarah Rees.
\newblock Sofic groups: graph products and graphs of groups.
\newblock {\em Pacific J. Math.}, 271(1):53--64, 2014.

\bibitem[CL15]{CapLup}
Valerio Capraro and Martino Lupini.
\newblock {\em Introduction to sofic and hyperlinear groups and {C}onnes'
  embedding conjecture}, volume 2136 of {\em Lecture Notes in Mathematics}.
\newblock Springer, Cham, 2015.
\newblock With an appendix by Vladimir Pestov.

\bibitem[Con76]{Connes}
A.~Connes.
\newblock Classification of injective factors. {C}ases {$II_{1},$} {$II_{\infty
  },$} {$III_{\lambda },$} {$\lambda \not=1$}.
\newblock {\em Ann. of Math. (2)}, 104(1):73--115, 1976.

\bibitem[DKP14]{DKP}
Ken Dykema, David Kerr, and Mika\"el Pichot.
\newblock Sofic dimension for discrete measured groupoids.
\newblock {\em Trans. Amer. Math. Soc.}, 366(2):707--748, 2014.

\bibitem[EL10]{ElekLip}
G\'abor Elek and G\'abor Lippner.
\newblock Sofic equivalence relations.
\newblock {\em J. Funct. Anal.}, 258(5):1692--1708, 2010.

\bibitem[ES04]{ESZKaplansky}
G\'abor Elek and Endre Szab\'o.
\newblock Sofic groups and direct finiteness.
\newblock {\em J. Algebra}, 280(2):426--434, 2004.

\bibitem[ES05]{ElekSzaboDeterminant}
G\'abor Elek and Endre Szab\'o.
\newblock Hyperlinearity, essentially free actions and {$L^2$}-invariants.
  {T}he sofic property.
\newblock {\em Math. Ann.}, 332(2):421--441, 2005.

\bibitem[ES06]{ESZ1}
G\'abor Elek and Endre Szab\'o.
\newblock On sofic groups.
\newblock {\em J. Group Theory}, 9(2):161--171, 2006.

\bibitem[ES11]{ESZ2}
G\'abor Elek and Endre Szab\'o.
\newblock Sofic representations of amenable groups.
\newblock {\em Proc. Amer. Math. Soc.}, 139(12):4285--4291, 2011.

\bibitem[GR08]{Glebsky-Rivera}
Lev Glebsky and Luis~Manuel Rivera.
\newblock Sofic groups and profinite topology on free groups.
\newblock {\em J. Algebra}, 320(9):3512--3518, 2008.

\bibitem[Gro99]{GromovSofic}
M.~Gromov.
\newblock Endomorphisms of symbolic algebraic varieties.
\newblock {\em J. Eur. Math. Soc. (JEMS)}, 1(2):109--197, 1999.

\bibitem[HR17]{HR}
Derek~F. Holt and Sarah Rees.
\newblock Some closure results for {$\mathcal{C}$}-approximable groups.
\newblock {\em Pacific J. Math.}, 287(2):393--409, 2017.

\bibitem[II09]{IliIwa}
Kei-ichiro Iima and Ryo Iwamatsu.
\newblock On the {J}ordan decomposition of tensored matrices of {J}ordan
  canonical forms.
\newblock {\em Math. J. Okayama Univ.}, 51:133--148, 2009.

\bibitem[KL11]{KLi}
David Kerr and Hanfeng Li.
\newblock Entropy and the variational principle for actions of sofic groups.
\newblock {\em Invent. Math.}, 186(3):501--558, 2011.

\bibitem[L\"02]{Luck}
Wolfgang L\"uck.
\newblock {\em {$L^2$}-invariants: theory and applications to geometry and
  {$K$}-theory}, volume~44 of {\em Ergebnisse der Mathematik und ihrer
  Grenzgebiete. 3. Folge. A Series of Modern Surveys in Mathematics [Results in
  Mathematics and Related Areas. 3rd Series. A Series of Modern Surveys in
  Mathematics]}.
\newblock Springer-Verlag, Berlin, 2002.

\bibitem[MV]{MatrsinVlass}
A.~Martsinkovksky and A.~Vlassov.
\newblock The representation rings of $k[x]$.
\newblock {\tt http://mathserver.neu.edu/\~{}martsinkovsky/GreenExcerpt.pdf}.

\bibitem[Pa11]{LPaun}
Liviu P\u~aunescu.
\newblock On sofic actions and equivalence relations.
\newblock {\em J. Funct. Anal.}, 261(9):2461--2485, 2011.

\bibitem[Pes08]{Pestov}
Vladimir~G. Pestov.
\newblock Hyperlinear and sofic groups: a brief guide.
\newblock {\em Bull. Symbolic Logic}, 14(4):449--480, 2008.

\bibitem[Pop14]{PoppArg}
Sorin Popa.
\newblock Independence properties in subalgebras of ultraproduct {$\rm II_1$}
  factors.
\newblock {\em J. Funct. Anal.}, 266(9):5818--5846, 2014.

\bibitem[Ra08]{Radul}
Florin R\u~adulescu.
\newblock The von {N}eumann algebra of the non-residually finite {B}aumslag
  group {$\langle a,b|ab^3a^{-1}=b^2\rangle$} embeds into {$R^\omega$}.
\newblock In {\em Hot topics in operator theory}, volume~9 of {\em Theta Ser.
  Adv. Math.}, pages 173--185. Theta, Bucharest, 2008.

\bibitem[Sal15]{Sale_Magnus}
Andrew~W. Sale.
\newblock Metric behaviour of the {M}agnus embedding.
\newblock {\em Geom. Dedicata}, 176:305--313, 2015.

\bibitem[VG97]{VerGor}
A.~M. Vershik and E.~I. Gordon.
\newblock Groups that are locally embeddable in the class of finite groups.
\newblock {\em Algebra i Analiz}, 9(1):71--97, 1997.

\bibitem[Wei00]{WeissSofic}
Benjamin Weiss.
\newblock Sofic groups and dynamical systems.
\newblock {\em Sankhy\=a Ser. A}, 62(3):350--359, 2000.
\newblock Ergodic theory and harmonic analysis (Mumbai, 1999).

\end{thebibliography}

\vspace{8mm}
\begin{minipage}{0.45\textwidth}{\small 
		\begin{flushleft}
			Ben Hayes\\
			University of Virginia\\
			Charlottesville, VA 22904, USA\\
			\emph{e-mail:} \texttt{brh5c@virginia.edu}\\[8mm]
	\end{flushleft}}
\end{minipage}\hspace{0.1\textwidth}
\begin{minipage}{0.45\textwidth}{\small
		\begin{flushleft}
			Andrew Sale\\
			Cornell University\\
			Ithaca, NY 14853, USA\\
			\emph{e-mail:} \texttt{andrew.sale@some.oxon.org}\\[8mm]
	\end{flushleft}}
\end{minipage}

\end{document}